\documentclass{amsart}
\usepackage{amsmath}
\usepackage{amssymb}
\usepackage{amsfonts}

\newtheorem{theorem}{Theorem}
\theoremstyle{plain}

\newtheorem{definition}{Definition}

\numberwithin{equation}{section}
\input{tcilatex}

\begin{document}
\title[ Some norm inequalities]{Some norm inequalities for commutators
generated by the Riesz potentials on homogeneous variable exponent
Herz-Morrey-Hardy spaces}
\author{FER\.{I}T G\"{U}RB\"{U}Z}
\address{Department of Mathematics, K\i rklareli University, K\i rklareli
39100, T\"{u}rkiye }
\email{feritgurbuz@klu.edu.tr}
\urladdr{}
\thanks{}
\curraddr{ }
\urladdr{}
\thanks{}
\date{}
\subjclass{Primary 46E35; Secondary 42B25, 42B35.}
\keywords{Riesz potential, variable exponent, commutator, homogeneous
Morrey-Herz-Hardy space, atomic decomposition}
\dedicatory{}
\thanks{}

\begin{abstract}
In harmonic analysis, the studies of inequalities of classical operators (=
singular, maximal, Riesz potentials etc.) in various function spaces have a
very important place. The maturation of many topics in the field of harmonic
analysis, as a result of various needs and developments to respond to the
problems of the time, has also led to the emergence of many studies and
works on these topics. In \cite{Gurbuz1}, under some conditions, the
boundedness of Riesz potential on homogeneous variable exponent
Herz-Morrey-Hardy spaces has been given. Inspired by the work of \cite%
{Gurbuz1}, in this work, by the atomic decompositions, we obtain the
boundedness of commutators generated by the Riesz potentials on homogeneous
variable exponent Herz-Morrey-Hardy spaces.
\end{abstract}

\maketitle

\section{Introduction}

In the last two decades it has become clear that classical function spaces
are no longer adequate to solve a number of contemporary problems that arise
naturally in various mathematical models of applied sciences. Therefore, it
has become necessary to introduce and study new nonstandard function spaces
from various points of view. Therefore, the study of variable exponent
function spaces has attracted considerable attention and has made
significant progress in various branches of mathematics, including real
analysis, partial differential equations (PDEs), and applied mathematics.
These spaces are an important generalization of classical function spaces
and were developed to study problems involving variable exponent functions.
These spaces eliminate several boundednesses of classical function spaces
and allow generalized differential and integral operators to be considered
in a broader context.

Herz spaces were initially introduced by Herz \cite{Herz} in his paper to
investigate the absolute convergence of Fourier transforms. Subsequently,
these spaces have found significant applications in diverse branches of
applied mathematics. Moreover, Herz \cite{Herz} gave a new class of space
are called the Herz spaces, which characterized certain properties of some
classical functions. Since Fefferman and Stein \cite{Fefferman and Stein}
gave a real variable characterization of Hardy space, many important
achievements have been made around Hardy space and its characterization as
the main representative.

In recent years, many Herz type spaces have emerged. One of the important
problems on Herz type spaces is boundedness of some integral operators. In
this work, only Herz-Morrey-Hardy spaces (which will be defined in the next
section) will be discussed. In this context, Herz-Morrey-Hardy spaces are
the generalized version of Herz-Hardy spaces. We will obtain some new
results in these spaces. In 2015, Xu and Yang \cite{Xu} first introduced
Herz-Morrey-Hardy spaces with variable exponents, secondly gave their atomic
decompositions. In 2016, Xu and Yang \cite{XU} gave the molecular
decomposition of variable exponent Herz-Morrey-Hardy spaces. In 2024, by
using the atomic decomposition for the domain Hardy space, G\"{u}rb\"{u}z 
\cite{Gurbuz1} showed the boundedness of Riesz potential on homogeneous
variable exponent Herz-Morrey-Hardy space. As a continuation of \cite%
{Gurbuz1}, we will discuss the commutator in this space in the third part of
this article.

\section{Basic Definitions and Previous Results}

Before giving the main result of this work, we first recall some elementary
facts, notations and necessary definitions that will be needed.

Define $\mathcal{P}\left( {\mathbb{R}^{n}}\right) $ to be the set of $%
p\left( \cdot \right) :{\mathbb{R}^{n}}\rightarrow \left[ 1,\infty \right) $
when $1\leq p_{-}:=\limfunc{essinf}\limits_{x\in {\mathbb{R}^{n}}}p\left(
x\right) $ and $p_{+}:=\limfunc{esssup}\limits_{x\in {\mathbb{R}^{n}}%
}p\left( x\right) <\infty $. The exponent $p^{\prime }\left( \cdot \right) $
means the conjugate of $p\left( \cdot \right) $, that is $\frac{1}{p\left(
x\right) }+\frac{1}{p^{\prime }\left( x\right) }=1$ holds. Let $p\left(
\cdot \right) \in \mathcal{P}\left( {\mathbb{R}^{n}}\right) $. Variable
exponent Lebesgue space $L^{p\left( \cdot \right) }\left( {\mathbb{R}^{n}}%
\right) $ is defined by%
\begin{equation*}
\left \Vert f\right \Vert _{L^{p\left( \cdot \right) }\left( {\mathbb{R}^{n}}%
\right) }:=\inf \left \{ \eta >0:\dint \limits_{{\mathbb{R}^{n}}}\left( 
\frac{\left \vert f\left( x\right) \right \vert }{\eta }\right) ^{p\left(
x\right) }dx\leq 1\right \} <\infty .
\end{equation*}

Let $f\in L_{loc}\left( {\mathbb{R}^{n}}\right) $. The Hardy-Littlewood
maximal operator $\mathcal{M}$ is defined by%
\begin{equation*}
\mathcal{M}f\left( x\right) :=\sup_{r>0}\frac{1}{r^{n}}\int \limits_{B\left(
x,r\right) }\left \vert f\left( y\right) \right \vert dy,\qquad \forall x\in 
{\mathbb{R}^{n},}
\end{equation*}%
where and follows $B(x,r)=\left \{ y\in {\mathbb{R}^{n}:}\left \vert
x-y\right \vert <r\right \} $ is the open ball centered at $x$ with radius $r
$ and $L_{loc}\left( {\mathbb{R}^{n}}\right) $ is the collection of all
locally integrable functions on ${\mathbb{R}^{n}}$. $\mathcal{B}\left( {%
\mathbb{R}^{n}}\right) $ is the set of $p\left( \cdot \right) \in \mathcal{P}%
\left( {\mathbb{R}^{n}}\right) $ satisfying the condition that $\mathcal{M}$
is bounded on $L^{p\left( \cdot \right) }\left( {\mathbb{R}^{n}}\right) $ as 
$p\left( \cdot \right) \in \mathcal{B}\left( {\mathbb{R}^{n}}\right) $ in 
\cite{Capone}. We claim that the letter $C$ stands for a positive constant
that, which may vary from line to line. The expression $f\lesssim g$ means $%
f\leqslant Cg$, and $f\thickapprox g$ means $f\lesssim g\lesssim f$.
Throughout this work, for simplicity, we denote $L^{p\left( \cdot \right)
}\left( {\mathbb{R}^{n}}\right) $ by $L^{p\left( \cdot \right) }$ and
similarly $B(x,r)$ by $B$. We will use the following results:

Let $\forall x,y\in \mathbb{%
\mathbb{R}
}^{n}$ and $C>0$. If $p\left( \cdot \right) \in \mathcal{P}\left( {\mathbb{R}%
^{n}}\right) $ satisfies the requirement given below%
\begin{equation}
\left \vert p\left( y\right) -p\left( x\right) \right \vert \leq \frac{-C}{%
\ln \left( \left \vert x-y\right \vert \right) },\qquad \text{ }\left \vert
x-y\right \vert \leq \frac{1}{2}  \label{11*}
\end{equation}%
\begin{equation}
\left \vert p\left( y\right) -p\left( x\right) \right \vert \leq \frac{C}{%
\ln \left( e+\left \vert x\right \vert \right) },\qquad \left \vert x\right
\vert \leq \left \vert y\right \vert ,  \label{12*}
\end{equation}%
then $p\left( \cdot \right) \in \mathcal{B}\left( {\mathbb{R}^{n}}\right) $
in \cite{Izuki1}.

For $p(\cdot )\in \mathcal{P}(\mathbb{R}^{n})$, $f\in L^{p\left( \cdot
\right) }\left( {\mathbb{R}^{n}}\right) $ and $g\in L^{p^{\prime }\left(
\cdot \right) }\left( {\mathbb{R}^{n}}\right) $, H\"{o}lder's inequality on
variable exponent Lebesgue spaces holds in the form 
\begin{equation}
\left \vert \dint \limits_{%
\mathbb{R}
^{n}}f\left( x\right) g\left( x\right) dx\right \vert \leq \dint \limits_{%
\mathbb{R}
^{n}}\left \vert f\left( x\right) g\left( x\right) \right \vert dx\leq
r_{p}\left \Vert f\right \Vert _{L^{p\left( \cdot \right) }}\left \Vert
g\right \Vert _{L^{p^{\prime }\left( \cdot \right) }},\text{ }r_{p}=1+\frac{1%
}{p_{-}}-\frac{1}{p_{+}},  \label{3}
\end{equation}%
see Theorem 2.1 in \cite{Kovacik}.

For $0<\beta <n$, the Riesz potential 
\begin{equation*}
I^{\beta }f(x)=\int \limits_{%
\mathbb{R}
^{n}}\frac{f(y)}{|x-y|^{n-\beta }}dy
\end{equation*}%
has many applications in harmonic analysis. For example, suppose $p(\cdot
)\in \mathcal{P}(\mathbb{R}^{n})$ satisfies (\ref{11*}) and (\ref{12*}).
Then, 
\begin{equation*}
\left \Vert I^{\beta }f\right \Vert _{L^{p_{2}\left( \cdot \right)
}}\lesssim \left \Vert f\right \Vert _{L^{p_{1}\left( \cdot \right) }}
\end{equation*}%
is valid such that $0<\beta <\frac{n}{p_{+}}$ and $\frac{1}{p_{2}\left(
\cdot \right) }=\frac{1}{p_{1}\left( \cdot \right) }-\frac{\beta }{n}$ (see 
\cite{Capone}). In 2024, G\"{u}rb\"{u}z \cite{Gurbuz1} obtained the
boundedness of the operator $I^{\beta }$ on homogeneous variable exponent
Herz-Morrey-Hardy spaces under some conditions. In addition, with \textbf{\ }%
$b\in L_{loc}\left( {\mathbb{R}^{n}}\right) $ ve $0<\beta <n$, in this work,
we can mainly focus on the the commutator generated by $I^{\beta }$ and $b$
is defined as follows:

\begin{eqnarray}
\left[ b,I^{\beta }\right] f\left( x\right) &=&b\left( x\right) I^{\beta
}f\left( x\right) -I^{\beta }\left( bf\right) \left( x\right)  \notag \\
&=&\int \limits_{\mathbb{R}^{n}}\frac{\left[ b\left( x\right) -b\left(
y\right) \right] }{\left \vert x-y\right \vert ^{n-\beta }}f\left( y\right)
dy.  \label{e1}
\end{eqnarray}%
A locally integrable function $b$ is said to belong to bounded mean
oscillation space $\left( BMO(\mathbb{R}^{n})\right) $, if it satisfies

\begin{equation*}
\Vert b\Vert _{BMO}:=\sup_{B}\frac{1}{|B|}\int
\limits_{B}|b(y)-b_{B}|\,dy<\infty ,
\end{equation*}%
where $B$ denotes the ball centered at $x\in \mathbb{R}^{n}$ and radius of $%
r>0$, and $b_{B}$ denotes the average of $b$ on $B$, that is, $%
b_{B}:=|B|^{-1}\int \limits_{B}b(t)dt$.

{Let{\ $b\in $ }}$BMO(\mathbb{R}^{n})${, }$p(\cdot )\in \mathcal{P}(\mathbb{R%
}^{n})$ {{and for $j,i\in \mathbb{Z}$ with $j>i$, we have the following
inequalities: 
\begin{equation*}
C^{-1}\Vert b\Vert _{BMO}\leq \sup_{B:Ball}\frac{1}{\Vert \chi _{B}\Vert
_{L^{p(\cdot )}}}\Vert (b-b_{B})\chi _{B}\Vert _{L^{p(\cdot )}}\leq C\Vert
b\Vert _{BMO}
\end{equation*}%
\begin{equation}
\Vert (b-b_{B_{i}})\chi _{B_{j}}\Vert _{L^{p(\cdot )}}\leq C(j-i)\Vert
b\Vert _{BMO}\Vert \chi _{B_{j}}\Vert _{L^{p(\cdot )}}  \label{l3}
\end{equation}%
(see \cite{Izuki1}).}}

Assume that $p(\cdot )\in \mathcal{P}(\mathbb{R}^{n})$ satisfies (\ref{11*})
and (\ref{12*}), $0<\beta <\frac{n}{p_{+}}$, and define $p_{2}\left( \cdot
\right) $ as $\frac{1}{p_{2}\left( \cdot \right) }=\frac{1}{p_{1}\left(
\cdot \right) }-\frac{\beta }{n}$. Then, 
\begin{equation}
\left \Vert \left[ b,I^{\beta }\right] f\right \Vert _{L^{p_{2}\left( \cdot
\right) }}\lesssim \Vert b\Vert _{BMO}\left \Vert f\right \Vert
_{L^{p_{1}\left( \cdot \right) }}  \label{14}
\end{equation}%
is satisfied (see \cite{Izuki1}).

When $p\left( \cdot \right) \in \mathcal{B}\left( {\mathbb{R}^{n}}\right) $,
then%
\begin{equation}
\frac{\left \Vert \chi _{S}\right \Vert _{L^{p\left( \cdot \right) }}}{\left
\Vert \chi _{B}\right \Vert _{L^{p\left( \cdot \right) }}}\leq C\left( \frac{%
\left \vert S\right \vert }{\left \vert B\right \vert }\right) ^{\delta
_{1}},\frac{\left \Vert \chi _{S}\right \Vert _{L^{p^{\prime }\left( \cdot
\right) }\left( 
\mathbb{R}
^{n}\right) }}{\left \Vert \chi _{B}\right \Vert _{L^{p^{\prime }\left(
\cdot \right) }\left( 
\mathbb{R}
^{n}\right) }}\leq C\left( \frac{\left \vert S\right \vert }{\left \vert
B\right \vert }\right) ^{\delta _{2}},S\subset B,  \label{1}
\end{equation}%
and{\large {\ }}%
\begin{equation}
\Vert \chi _{B}\Vert _{L^{p(\cdot )}}\Vert \chi _{B}\Vert _{L^{p^{\prime
}(\cdot )}}\leq C\left \vert B\right \vert ,  \label{5}
\end{equation}%
were proved in \cite{Izuki}, respectively.

By (\ref{1}) and (\ref{5}), we can conclude that 
\begin{equation}
\Vert \chi _{i}\Vert _{L^{q(\cdot )}}\leq \Vert \chi _{B_{i}}\Vert
_{L^{q(\cdot )}}.  \label{6}
\end{equation}

Let $l\in \mathbb{%
\mathbb{Z}
}$, $B_{l}:=\{x\in \mathbb{R}^{n}:\left \vert x\right \vert \leq 2^{l}\}$, $%
D_{l}:=B_{l}\setminus B_{l-1}$, $\chi _{l}:=\chi _{D_{l}}$. Assume that the
symbol $%
\mathbb{N}
_{0}$ denotes the set of all nonnegative integers. For any $m\in 
\mathbb{N}
_{0}=%
\mathbb{N}
\cup \left \{ 0\right \} $, we define

\begin{equation*}
\tilde{\chi}_{m}:=\left \{ 
\begin{array}{ccc}
\chi _{D_{m}} & , & m\geq 1 \\ 
\chi _{B_{0}} & , & m=0%
\end{array}%
\right. \cdot
\end{equation*}

Now let's give the definition of the homogeneous variable exponent
Herz-Morrey spaces.

\begin{definition}
Let $0<q\leq \infty $, $p\left( \cdot \right) \in \mathcal{P}\left( {\mathbb{%
R}^{n}}\right) $ and $0\leq \lambda <\infty $. Let $\alpha \left( \cdot
\right) $ also be a bounded real-valued measurable function on $\mathbb{R}%
^{n}\left( \text{that is, }\alpha \left( \cdot \right) \in L^{\infty }\left( 
\mathbb{R}^{n}\right) \right) $. Then, the homogeneous variable exponent
Herz-Morrey space $M\dot{K}_{p(\cdot ),\lambda }^{\alpha \left( \cdot
\right) ,q}(\mathbb{R}^{n})$ is defined by%
\begin{equation*}
M\dot{K}_{p(\cdot ),\lambda }^{\alpha \left( \cdot \right) ,q}(\mathbb{R}%
^{n}):=\left \{ f\in L_{loc}^{p(\cdot )}\left( \mathbb{R}^{n}\setminus \{0\}
\right) :\Vert f\Vert _{M\dot{K}_{p(\cdot ),\lambda }^{\alpha \left( \cdot
\right) ,q}(\mathbb{R}^{n})}<\infty \right \} ,
\end{equation*}%
where%
\begin{equation*}
\Vert f\Vert _{M\dot{K}_{p(\cdot ),\lambda }^{\alpha \left( \cdot \right)
,q}(\mathbb{R}^{n})}:=\sup_{L\in \mathbb{%
\mathbb{Z}
}}2^{-L\lambda }\left( \sum \limits_{l=-\infty }^{L}\Vert 2^{l\alpha \left(
\cdot \right) }f\chi _{l}\Vert _{L^{p(\cdot )}}^{q}\right) ^{\frac{1}{q}}
\end{equation*}%
and when $q=\infty $, there is the usual modification.
\end{definition}

Now before giving the definition of homogeneous variable exponent
Herz-Morrey-Hardy space $HM\dot{K}_{p(\cdot ),\lambda }^{\alpha \left( \cdot
\right) ,q}(\mathbb{R}^{n})$, we need to recall some notations.

We assume that $\mathcal{S}\left( 
\mathbb{R}
^{n}\right) $ be the Schwartz space of all rapidly decreasing infinitely
differentiable functions on $%
\mathbb{R}
^{n}$, and $\mathcal{S}^{\prime }\left( 
\mathbb{R}
^{n}\right) $ be the dual space of $\mathcal{S}\left( 
\mathbb{R}
^{n}\right) $. Assume $G_{N}f$ is the grand maximal function of $f$ defined
by%
\begin{equation*}
G_{N}f\left( x\right) :=\sup_{\phi \in \mathcal{A}_{N}}\left \vert \phi
_{\nabla }^{\ast }\left( f\right) \left( x\right) \right \vert ,\qquad x\in 
\mathbb{R}
^{n},
\end{equation*}%
where $\mathcal{A}_{N}$ and $\phi _{\nabla }^{\ast }$ are defined in \cite%
{Xu}, respectively.

In 2015, Xu and Yang \cite{Xu} introduced the homogeneous variable exponent
Herz-Morrey-Hardy spaces $HM\dot{K}_{p(\cdot ),\lambda }^{\alpha \left(
\cdot \right) ,q}(\mathbb{R}^{n})$ as follows:

\begin{definition}
Let $\alpha \left( \cdot \right) \in L^{\infty }\left( \mathbb{R}^{n}\right) 
$, $0<q\leq \infty $, $p\left( \cdot \right) \in \mathcal{P}\left( {\mathbb{R%
}^{n}}\right) $, $0\leq \lambda <\infty $ and $N>n+1$. Then, the homogeneous
variable exponent Herz-Morrey-Hardy space $HM\dot{K}_{p(\cdot ),\lambda
}^{\alpha \left( \cdot \right) ,q}(\mathbb{R}^{n})$ is defined by%
\begin{equation*}
HM\dot{K}_{p(\cdot ),\lambda }^{\alpha \left( \cdot \right) ,q}(\mathbb{R}%
^{n}):=\left \{ 
\begin{array}{c}
f\in \mathcal{S}^{\prime }\left( 
\mathbb{R}
^{n}\right) :\Vert f\Vert _{HM\dot{K}_{p(\cdot ),\lambda }^{\alpha \left(
\cdot \right) ,q}(\mathbb{R}^{n})}:= \\ 
\Vert G_{N}f\Vert _{M\dot{K}_{p(\cdot ),\lambda }^{\alpha \left( \cdot
\right) ,q}(\mathbb{R}^{n})}<\infty%
\end{array}%
\right \} .
\end{equation*}
\end{definition}

We can say that $G_{N}f\left( x\right) \leq C\mathcal{M}f\left( x\right)
\left( C>0\right) $ for $x\in 
\mathbb{R}
^{n}$ (see Proposition in Page 57 in \cite{Stein}).

Now, we give the notion of atoms.

\begin{definition}
\label{Definition}Let $p\left( \cdot \right) \in \mathcal{P}\left( \mathbb{R}%
^{n}\right) $, $\alpha \left( \cdot \right) \in L^{\infty }\left( \mathbb{R}%
^{n}\right) \cap \mathcal{P}_{0}^{\log }\left( 
\mathbb{R}
^{n}\right) \cap \mathcal{P}_{\infty }^{\log }\left( 
\mathbb{R}
^{n}\right) $, and nonnegative integer $s\geqslant \left[ \alpha
_{r}-n\delta _{2}\right] $; here $\alpha _{r}=\alpha (0)$, if $r<1$, and $%
\alpha _{r}=\alpha _{\infty }$, if $r\geqslant 1$, $n\delta _{2}\leq \alpha
_{r}<\infty $ and $\delta _{2}$ as in (\ref{1}).

$(i)$ If a function $a$ on $\mathbb{R}^{n}$satisfies%
\begin{eqnarray*}
\left( 1\right) \text{ supp }a &\subset &B\left( 0,2^{r}\right) \\
\left( 2\right) \text{ }\left \Vert a\right \Vert _{L^{p\left( \cdot \right)
}} &\leq &\left \vert B\left( 0,2^{r}\right) \right \vert ^{-\frac{\alpha
_{r}}{n}} \\
\left( 3\right) \text{ }\dint \limits_{%
\mathbb{R}
^{n}}a\left( x\right) x^{\beta }dx &=&0,\text{ }\left \vert \beta \right
\vert \leq s,
\end{eqnarray*}%
then it is called a central $(\alpha (\cdot ),p(\cdot ))$-atom.

$(ii)$ If a function $a$ on $\mathbb{R}^{n}$satisfies $\left( 2\right) $ and 
$\left( 3\right) $ above and condition given below:%
\begin{equation*}
\left( 1\right) \text{ supp }a\subset B\left( 0,2^{r}\right) ,\text{ }r\geq
1,
\end{equation*}%
then it is called a central $(\alpha (\cdot ),p(\cdot ))$-atom.of restricted
type, where follows $\mathcal{P}_{0}^{\log }\left( 
\mathbb{R}
^{n}\right) $ and $\mathcal{P}_{\infty }^{\log }\left( 
\mathbb{R}
^{n}\right) $ are the set of log-H\"{o}lder continuous functions at origin
and the set of log-H\"{o}lder continuous functions at infinity,
respectively; for their definitions see \cite{Gurbuz1}.
\end{definition}

\section{Main Result}

The following theorem is our main result.

\begin{theorem}
Let $0<q_{1}\leq q_{2}<\infty $, $0\leq \lambda <\infty $, $0<\beta <n$ and $%
\alpha \left( \cdot \right) \in L^{\infty }\left( \mathbb{R}^{n}\right) \cap 
\mathcal{P}_{0}^{\log }\left( 
\mathbb{R}
^{n}\right) \cap \mathcal{P}_{\infty }^{\log }\left( 
\mathbb{R}
^{n}\right) $ such that $2\lambda \leq \alpha \left( \cdot \right) $, $\beta
-n\delta _{2}<\alpha \left( 0\right) $, $\beta -n\delta _{2}<\alpha _{\infty
}<\infty $ with $\delta _{1}$, $\delta _{2}\in \left( 0,1\right) $
satisfying (\ref{1}). Let also $\left[ b,I^{\beta }\right] $ be defined as
in (\ref{e1}). Assume that $p(\cdot ),p_{1}\left( \cdot \right) ,p_{2}\left(
\cdot \right) \in \mathcal{P}(\mathbb{R}^{n})$ satisfy (\ref{11*}) and (\ref%
{12*}) and define $p_{1}\left( \cdot \right) $ by $\frac{1}{p_{2}\left(
\cdot \right) }=\frac{\beta }{n}-\frac{1}{p_{1}\left( \cdot \right) }$ for $%
1<p_{1}\left( \cdot \right) <p_{2}\left( \cdot \right) <\infty $. \i f {$%
b\in $}$\Vert b\Vert _{BMO}$, then the following inequality holds:%
\begin{equation}
\left \Vert \left[ b,I^{\beta }\right] \left( f\right) \right \Vert _{M\dot{K%
}_{p_{2}(\cdot ),\lambda }^{\alpha \left( \cdot \right) ,q_{2}}(\mathbb{R}%
^{n})}\lesssim \Vert b\Vert _{BMO}\left \Vert f\right \Vert _{HM\dot{K}%
_{p_{1}(\cdot ),\lambda }^{\alpha \left( \cdot \right) ,q_{1}}(\mathbb{R}%
^{n})}  \label{100}
\end{equation}%
for all $f\in HM\dot{K}_{p_{1}(\cdot ),\lambda }^{\alpha \left( \cdot
\right) ,q_{1}}(\mathbb{R}^{n})$.
\end{theorem}

\begin{proof}
Let $f\in HM\dot{K}_{p_{1}(\cdot ),\lambda }^{\alpha \left( \cdot \right)
,q_{1}}(\mathbb{R}^{n})$. By Theorem 2 in \cite{Gurbuz1}, take $a_{l}=a\cdot
\chi _{l}=a\cdot \chi _{B_{l}}$ for each $l\in 
\mathbb{Z}
$, we have 
\begin{equation*}
f=\dsum \limits_{l=0}^{\infty }\lambda _{l}a=\dsum \limits_{l=-\infty
}^{\infty }\lambda _{l}a_{l},
\end{equation*}%
where each $a_{l}$ is a central $(\alpha (\cdot ),q_{1}(\cdot ))$-atom with
support contained in $B_{l}$, and%
\begin{equation*}
\Vert f\Vert _{HM\dot{K}_{p_{1}(\cdot ),\lambda }^{\alpha \left( \cdot
\right) ,q_{1}}(\mathbb{R}^{n})}\approx \inf \sup \limits_{L\in 
\mathbb{Z}
}2^{-L\lambda }\left( \sum \limits_{l=-\infty }^{L}\left \vert \lambda
_{l}\right \vert ^{q_{1}}\right) ^{\frac{1}{q_{1}}}.
\end{equation*}

For convenience, we show 
\begin{equation*}
\Psi =\sup \limits_{L\in 
\mathbb{Z}
}2^{-L\lambda q_{1}}\sum \limits_{l=-\infty }^{L}\left \vert \lambda
_{l}\right \vert ^{q_{1}}.
\end{equation*}%
By (2.8) in \cite{Gurbuz1}, we have 
\begin{eqnarray*}
\left \Vert \left[ b,I^{\beta }\right] \left( f\right) \right \Vert _{M\dot{K%
}_{p_{2}(\cdot ),\lambda }^{\alpha \left( \cdot \right) ,q_{2}}(\mathbb{R}%
^{n})}^{q_{1}} &\approx &\max \left \{ 
\begin{array}{c}
\sup \limits_{L\leq 0,L\in 
\mathbb{Z}
}2^{-L\lambda q_{1}}\left( \sum \limits_{k=-\infty }^{L}2^{kq_{1}\alpha
\left( 0\right) }\left \Vert \left[ b,I^{\beta }\right] \left( f\right) \chi
_{k}\right \Vert _{L^{p_{2}(\cdot )}}^{q_{1}}\right) , \\ 
\sup \limits_{L>0,L\in 
\mathbb{Z}
}\left[ 
\begin{array}{c}
2^{-L\lambda q_{1}}\left( \sum \limits_{k=-\infty }^{-1}2^{kq_{1}\alpha
\left( 0\right) }\left \Vert \left[ b,I^{\beta }\right] \left( f\right) \chi
_{k}\right \Vert _{L^{p_{2}(\cdot )}}^{q_{1}}\right) \\ 
+2^{-L\lambda q_{1}}\left( \sum \limits_{k=0}^{L}2^{kq_{1}\alpha _{\infty
}}\left \Vert \left[ b,I^{\beta }\right] \left( f\right) \chi _{k}\right
\Vert _{L^{p_{2}(\cdot )}}^{q_{1}}\right)%
\end{array}%
\right]%
\end{array}%
\right \} \\
&\lesssim &\left \{ X,Y+Z\right \} ,
\end{eqnarray*}%
where%
\begin{eqnarray*}
X &:&=\sup \limits_{L\leq 0,L\in 
\mathbb{Z}
}2^{-L\lambda q_{1}}\left( \sum \limits_{k=-\infty }^{L}2^{kq_{1}\alpha
\left( 0\right) }\left \Vert \left[ b,I^{\beta }\right] \left( f\right) \chi
_{k}\right \Vert _{L^{p_{2}(\cdot )}}^{q_{1}}\right) \\
Y &:&=\sum \limits_{k=-\infty }^{-1}2^{kq_{1}\alpha \left( 0\right) }\left
\Vert \left[ b,I^{\beta }\right] \left( f\right) \chi _{k}\right \Vert
_{L^{p_{2}(\cdot )}}^{q_{1}} \\
Z &:&=\sup \limits_{L>0,L\in 
\mathbb{Z}
}2^{-L\lambda q_{1}}\left( \sum \limits_{k=0}^{L}2^{kq_{1}\alpha _{\infty
}}\left \Vert \left[ b,I^{\beta }\right] \left( f\right) \chi _{k}\right
\Vert _{L^{p_{2}(\cdot )}}^{q_{1}}\right) .
\end{eqnarray*}%
To complete the proof of main result, we only need prove that $X,Y,Z\lesssim
\Psi $. To do this we estimate $X,Y,Z$ step by step. Indeed, we have%
\begin{eqnarray*}
X &=&\sup \limits_{L\leq 0,L\in 
\mathbb{Z}
}2^{-L\lambda q_{1}}\left( \sum \limits_{k=-\infty }^{L}2^{kq_{1}\alpha
\left( 0\right) }\left \Vert \left[ b,I^{\beta }\right] \left( f\right) \chi
_{k}\right \Vert _{L^{p_{2}(\cdot )}}^{q_{1}}\right) \\
&\leq &\sup \limits_{L\leq 0,L\in 
\mathbb{Z}
}2^{-L\lambda q_{1}}\sum \limits_{k=-\infty }^{L}2^{kq_{1}\alpha \left(
0\right) }\left( \sum \limits_{l=k}^{\infty }\left \vert \lambda _{l}\right
\vert \left \Vert \left( \left[ b,I^{\beta }\right] \left( a_{l}\right)
\right) \chi _{k}\right \Vert _{L^{p_{2}(\cdot )}}\right) ^{q_{1}} \\
&&+\sup \limits_{L\leq 0,L\in 
\mathbb{Z}
}2^{-L\lambda q_{1}}\sum \limits_{k=-\infty }^{L}2^{kq_{1}\alpha \left(
0\right) }\left( \sum \limits_{l=-\infty }^{k-1}\left \vert \lambda
_{l}\right \vert \left \Vert \left( \left[ b,I^{\beta }\right] \left(
a_{l}\right) \right) \chi _{k}\right \Vert _{L^{p_{2}(\cdot )}}\right)
^{q_{1}} \\
&:&=X_{1}+X_{2}.
\end{eqnarray*}%
\begin{eqnarray*}
Y &=&\sum \limits_{k=-\infty }^{-1}2^{kq_{1}\alpha \left( 0\right) }\left
\Vert \left[ b,I^{\beta }\right] \left( f\right) \chi _{k}\right \Vert
_{L^{p_{2}(\cdot )}}^{q_{1}} \\
&\leq &\sum \limits_{k=-\infty }^{-1}2^{kq_{1}\alpha \left( 0\right) }\left(
\sum \limits_{l=k}^{\infty }\left \vert \lambda _{l}\right \vert \left \Vert
\left( \left[ b,I^{\beta }\right] \left( a_{l}\right) \right) \chi
_{k}\right \Vert _{L^{p_{2}(\cdot )}}\right) ^{q_{1}} \\
&&+\sum \limits_{k=-\infty }^{-1}2^{kq_{1}\alpha \left( 0\right) }\left(
\sum \limits_{l=-\infty }^{k-1}\left \vert \lambda _{l}\right \vert \left
\Vert \left( \left[ b,I^{\beta }\right] \left( a_{l}\right) \right) \chi
_{k}\right \Vert _{L^{p_{2}(\cdot )}}\right) ^{q_{1}} \\
&:&=Y_{1}+Y_{2}.
\end{eqnarray*}%
\begin{eqnarray*}
Z &=&\sup \limits_{L>0,L\in 
\mathbb{Z}
}2^{-L\lambda q_{1}}\left( \sum \limits_{k=0}^{L}2^{kq_{1}\alpha _{\infty
}}\left \Vert \left[ b,I^{\beta }\right] \left( f\right) \chi _{k}\right
\Vert _{L^{p_{2}(\cdot )}}^{q_{1}}\right) \\
&\lesssim &\sup \limits_{L>0,L\in 
\mathbb{Z}
}2^{-L\lambda q_{1}}\sum \limits_{k=0}^{L}2^{kq_{1}\alpha _{\infty }}\left(
\sum \limits_{l=k}^{\infty }\left \vert \lambda _{l}\right \vert \left \Vert
\left( \left[ b,I^{\beta }\right] \left( a_{l}\right) \right) \chi
_{k}\right \Vert _{L^{p_{2}(\cdot )}}\right) ^{q_{1}} \\
&&+\sup \limits_{L>0,L\in 
\mathbb{Z}
}2^{-L\lambda q_{1}}\sum \limits_{k=0}^{L}2^{kq_{1}\alpha _{\infty }}\left(
\sum \limits_{l=-\infty }^{k-1}\left \vert \lambda _{l}\right \vert \left
\Vert \left( \left[ b,I^{\beta }\right] \left( a_{l}\right) \right) \chi
_{k}\right \Vert _{L^{p_{2}(\cdot )}}\right) ^{q_{1}} \\
&:&=Z_{1}+Z_{2}.
\end{eqnarray*}%
To proceed, we need a pointwise estimate for $\left[ b,I^{\beta }\right] $
on $F_{l}$, where by (\ref{3}), for any $l,k\in 
\mathbb{Z}
$, with $l\leq k-1$ and $y\in F_{l}$, then $\left \vert x-y\right \vert \sim
\left \vert x\right \vert $, $2\left \vert y\right \vert \leq \left \vert
x\right \vert $, we obtain 
\begin{eqnarray}
&&\left \vert \left[ b,I^{\beta }\right] \left( a_{j}\right) \chi _{k}\left(
x\right) \right \vert  \notag \\
&\leq &\int \limits_{F_{l}}\frac{\left \vert b\left( x\right) -b\left(
y\right) \right \vert }{|x-y|^{n-\beta }}\left \vert a_{l}\left( y\right)
\right \vert dy\cdot \chi _{k}\left( x\right)  \notag \\
&\lesssim &2^{k\left( \beta -n\right) }\left( \left \vert b\left( x\right)
-b_{B_{l}}\right \vert \int \limits_{F_{l}}\left \vert a_{l}\left( y\right)
\right \vert dy+\int \limits_{F_{l}}\left \vert b_{B_{l}}-b\left( y\right)
\right \vert \left \vert a_{l}\left( y\right) \right \vert dy\right) \cdot
\chi _{k}\left( x\right)  \notag \\
&\lesssim &2^{k\left( \beta -n\right) }\left \Vert a_{l}\right \Vert
_{L^{p_{1}(\cdot )}}  \notag \\
&&\times \left( \left \vert b\left( x\right) -b_{B_{l}}\right \vert \left
\Vert \chi _{l}\right \Vert _{L^{p_{1}^{\prime }(\cdot )}}+\left \Vert
\left( b-b_{B_{j}}\right) \chi _{l}\right \Vert _{L^{p_{1}^{\prime }(\cdot
)}}\right) \cdot \chi _{k}\left( x\right) .  \label{4}
\end{eqnarray}%
So applying (\ref{3}), (\ref{l3}), (\ref{14}), (\ref{1}), (\ref{6}) and (\ref%
{4}), we have%
\begin{eqnarray*}
\left \Vert \left[ b,I^{\beta }\right] \left( a_{l}\right) \chi _{k}\right
\Vert _{L^{p_{2}(\cdot )}} &\lesssim &2^{k\left( \beta -n\right) }\left
\Vert a_{l}\right \Vert _{L^{p_{1}(\cdot )}} \\
&&\times \left( \left \Vert \left( b-b_{B_{l}}\right) \chi _{k}\right \Vert
_{L^{p_{2}(\cdot )}}\left \Vert \chi _{l}\right \Vert _{L^{p_{1}^{\prime
}(\cdot )}}+\left \Vert \left( b-b_{B_{j}}\right) \chi _{l}\right \Vert
_{L^{p_{1}^{\prime }(\cdot )}}\left \Vert \chi _{k}\right \Vert
_{L^{p_{2}(\cdot )}}\right) \\
&\lesssim &2^{k\left( \beta -n\right) }\left \Vert a_{j}\right \Vert
_{L^{p_{1}(\cdot )}} \\
&&\times \left( \left( k-l\right) \left \Vert b\right \Vert _{BMO}\Vert \chi
_{B_{k}}\Vert _{L^{p_{2}(\cdot )}}\left \Vert \chi _{l}\right \Vert
_{L^{p_{1}^{\prime }(\cdot )}}+\left \Vert b\right \Vert _{BMO}\Vert \chi
_{B_{l}}\Vert _{L^{p_{1}^{\prime }(\cdot )}}\left \Vert \chi _{k}\right
\Vert _{L^{p_{2}(\cdot )}}\right) \\
&\lesssim &\left \Vert b\right \Vert _{BMO}2^{k\left( \beta -n\right) }\left
\Vert a_{l}\right \Vert _{L^{p_{1}(\cdot )}}\left( k-l\right) \Vert \chi
_{B_{k}}\Vert _{L^{p_{2}(\cdot )}}\left \Vert \chi _{l}\right \Vert
_{L^{p_{1}^{\prime }(\cdot )}} \\
&\lesssim &\left \Vert b\right \Vert _{BMO}2^{k\beta }\left( k-l\right)
\left \Vert a_{l}\right \Vert _{L^{p_{1}(\cdot )}}\left \Vert \chi
_{l}\right \Vert _{L^{p_{1}^{\prime }(\cdot )}}\left \Vert \chi
_{B_{k}}\right \Vert _{L^{p_{2}^{\prime }(\cdot )}}^{-1} \\
&\lesssim &\left \Vert b\right \Vert _{BMO}2^{k\beta }\left( k-l\right)
\left \Vert a_{l}\right \Vert _{L^{p_{1}(\cdot )}}\left \Vert \chi
_{l}\right \Vert _{L^{p_{1}^{\prime }(\cdot )}}\left \Vert \chi
_{B_{k}}\right \Vert _{L^{p_{2}^{\prime }(\cdot )}}^{-1}\frac{\left \Vert
\chi _{B_{l}}\right \Vert _{L^{p_{2}^{\prime }(\cdot )}}}{\left \Vert \chi
_{B_{k}}\right \Vert _{L^{p_{2}^{\prime }(\cdot )}}} \\
&\lesssim &\left \Vert b\right \Vert _{BMO}2^{k\beta }\left( k-l\right)
\left \Vert a_{l}\right \Vert _{L^{p_{1}(\cdot )}}\left \Vert \chi
_{B_{l}}\right \Vert _{L^{p_{1}^{\prime }(\cdot )}}\left \Vert \chi
_{B_{l}}\right \Vert _{L^{p_{2}^{\prime }(\cdot )}}^{-1}2^{n\delta
_{2}\left( l-k\right) } \\
&\lesssim &\left \Vert b\right \Vert _{BMO}2^{k\beta }\left( k-l\right)
\left \Vert a_{l}\right \Vert _{L^{p_{1}(\cdot )}}2^{\left( n-\beta \right)
l}\left \Vert \chi _{B_{l}}\right \Vert _{L^{p_{2}(\cdot )}}^{-1}\left \Vert
\chi _{B_{l}}\right \Vert _{L^{p_{2}^{\prime }(\cdot )}}^{-1}2^{n\delta
_{2}\left( l-k\right) } \\
&\lesssim &\left \Vert b\right \Vert _{BMO}2^{\left( n\delta _{2}-\beta
\right) \left( l-k\right) }\left( k-l\right) \left \Vert a_{l}\right \Vert
_{L^{p_{1}(\cdot )}}\left( 2^{-ln}\left \Vert \chi _{B_{l}}\right \Vert
_{L^{p_{2}(\cdot )}}\left \Vert \chi _{B_{l}}\right \Vert _{L^{p_{2}^{\prime
}(\cdot )}}\right) ^{-1} \\
&\lesssim &\left \Vert b\right \Vert _{BMO}2^{\left( \beta -n\delta
_{2}\right) \left( k-l\right) }\left( k-l\right) \left \Vert a_{l}\right
\Vert _{L^{p_{1}(\cdot )}} \\
&\lesssim &\left \Vert b\right \Vert _{BMO}\left( k-l\right) 2^{\left( \beta
-n\delta _{2}\right) \left( k-l\right) -\alpha _{l}l},
\end{eqnarray*}%
where in the sixth inequality we have used the following fact:

First, we know that 
\begin{equation*}
\chi _{B_{l}}\left( x\right) \lesssim 2^{-l\beta }\left[ b,I^{\beta }\right]
\left( \chi _{B_{l}}\right) \left( x\right) .
\end{equation*}%
Also, let $p_{1}\left( \cdot \right) \in \mathcal{P}\left( {\mathbb{R}^{n}}%
\right) $.and $0<\beta <\frac{n}{\left( p_{1}\right) _{+}}$. Then, by (\ref%
{14}) and using (\ref{5}), we obtain%
\begin{eqnarray*}
\left \Vert \chi _{B_{l}}\right \Vert _{L^{p_{2}(\cdot )}} &\lesssim
&2^{-l\beta }\left \Vert \left[ b,I^{\beta }\right] \left( \chi
_{B_{l}}\right) \right \Vert _{L^{p_{2}(\cdot )}} \\
&\lesssim &\left \Vert b\right \Vert _{BMO}2^{-l\beta }\left \Vert \chi
_{B_{l}}\right \Vert _{L^{p_{1}(\cdot )}} \\
&\lesssim &\left \Vert b\right \Vert _{BMO}2^{\left( n-\beta \right) l}\left
\Vert \chi _{B_{l}}\right \Vert _{L^{p_{1}^{\prime }(\cdot )}}^{-1}.
\end{eqnarray*}%
Thus, we get%
\begin{equation}
\left \Vert \left[ b,I^{\beta }\right] \left( a_{l}\right) \chi _{k}\right
\Vert _{L^{p_{2}(\cdot )}}\lesssim \left \Vert b\right \Vert _{BMO}\left(
k-l\right) 2^{\left( \beta -n\delta _{2}\right) \left( k-l\right) -\alpha
_{l}l}.  \label{12}
\end{equation}

To complete the proof, we consider $X,Y,Z$ into two cases $0<q_{1}\leq 1$
and $1<q_{1}<\infty $.

\textbf{Case 1: }$\left( 0<q_{1}\leq 1\right) $. In this case, we always use
(3.6) in \cite{Gurbuz1} and the convergence of a geometric series and
exchange order of summation and the convergence of geometric power series.

Firstly, since $\left[ b,I^{\beta }\right] $ is bounded from $L^{p_{1}\left(
\cdot \right) }$ to $L^{p_{2}\left( \cdot \right) }$ and using the condition 
$\left( 2\right) $ in Definition \ref{Definition}, then we get%
\begin{equation}
\left \Vert \left[ b,I^{\beta }\right] \left( a_{l}\right) \chi _{k}\right
\Vert _{L^{p_{2}(\cdot )}}\leq \left \Vert b\right \Vert _{BMO}\left \Vert
a_{l}\right \Vert _{L^{p_{1}(\cdot )}}\leq \left \Vert b\right \Vert
_{BMO}\left \vert B_{l}\right \vert ^{-\frac{\alpha _{j}}{n}}=\left \Vert
b\right \Vert _{BMO}2^{-l\alpha _{l}}.  \label{10}
\end{equation}

Thus, by (3.6) in \cite{Gurbuz1} and (\ref{10}), we obtain%
\begin{eqnarray*}
X_{1} &=&\sup \limits_{L\leq 0,L\in 
\mathbb{Z}
}2^{-L\lambda q_{1}}\sum \limits_{k=-\infty }^{L}2^{kq_{1}\alpha \left(
0\right) }\left( \sum \limits_{l=k}^{\infty }\left \vert \lambda _{l}\right
\vert \left \Vert \left( \left[ b,I^{\beta }\right] \left( a_{l}\right)
\right) \chi _{k}\right \Vert _{L^{p_{2}(\cdot )}}\right) ^{q_{1}} \\
&\lesssim &\sup \limits_{L\leq 0,L\in 
\mathbb{Z}
}2^{-L\lambda q_{1}}\sum \limits_{k=-\infty }^{L}2^{kq_{1}\alpha \left(
0\right) }\left( \sum \limits_{l=k}^{\infty }\left \vert \lambda _{l}\right
\vert \left \Vert b\right \Vert _{BMO}2^{-l\alpha _{l}}\right) ^{q_{1}} \\
&\lesssim &\left \Vert b\right \Vert _{BMO}^{q_{1}}\sup \limits_{L\leq
0,L\in 
\mathbb{Z}
}2^{-L\lambda q_{1}}\sum \limits_{k=-\infty }^{L}2^{kq_{1}\alpha \left(
0\right) }\left( \sum \limits_{l=k}^{-1}\left \vert \lambda _{l}\right \vert
^{q_{1}}2^{-l\alpha \left( 0\right) q_{1}}+\sum \limits_{l=0}^{\infty }\left
\vert \lambda _{l}\right \vert ^{q_{1}}2^{-l\alpha _{\infty }q_{1}}\right) \\
&\lesssim &\left \Vert b\right \Vert _{BMO}^{q_{1}}\sup \limits_{L\leq
0,L\in 
\mathbb{Z}
}2^{-L\lambda q_{1}}\sum \limits_{k=-\infty }^{L}\sum
\limits_{l=k}^{-1}\left \vert \lambda _{l}\right \vert ^{q_{1}}2^{\left(
k-l\right) \alpha \left( 0\right) q_{1}} \\
&&+\left \Vert b\right \Vert _{BMO}^{q_{1}}\sup \limits_{L\leq 0,L\in 
\mathbb{Z}
}2^{-L\lambda q_{1}}\sum \limits_{k=-\infty }^{L}2^{k\alpha \left( 0\right)
q_{1}}\sum \limits_{l=0}^{\infty }\left \vert \lambda _{l}\right \vert
^{q_{1}}2^{-l\alpha _{\infty }q_{1}} \\
&\lesssim &\left \Vert b\right \Vert _{BMO}^{q_{1}}\sup \limits_{L\leq
0,L\in 
\mathbb{Z}
}2^{-L\lambda q_{1}}\sum \limits_{l=-\infty }^{-1}\left \vert \lambda
_{l}\right \vert ^{q_{1}}\sum \limits_{k=-\infty }^{l}2^{\left( k-l\right)
\alpha \left( 0\right) q_{1}} \\
&&+\left \Vert b\right \Vert _{BMO}^{q_{1}}\sup \limits_{L\leq 0,L\in 
\mathbb{Z}
}\sum \limits_{l=0}^{\infty }2^{-l\lambda q_{1}}\left \vert \lambda
_{l}\right \vert ^{q_{1}}2^{\left( \lambda -\alpha _{\infty }\right)
lq_{1}}2^{-L\lambda q_{1}}\sum \limits_{k=-\infty }^{L}2^{k\alpha \left(
0\right) q_{1}} \\
&\leq &\left \Vert b\right \Vert _{BMO}^{q_{1}}\left( \sup \limits_{L\leq
0,L\in 
\mathbb{Z}
}2^{-L\lambda q_{1}}\sum \limits_{l=-\infty }^{L}\left \vert \lambda
_{l}\right \vert ^{q_{1}}+\sup \limits_{L\leq 0,L\in 
\mathbb{Z}
}2^{-L\lambda q_{1}}\sum \limits_{l=L}^{-1}\left \vert \lambda _{l}\right
\vert ^{q_{1}}\sum \limits_{k=-\infty }^{l}2^{\left( k-l\right) \alpha
\left( 0\right) q_{1}}\right) \\
&&+\left \Vert b\right \Vert _{BMO}^{q_{1}}\Psi \sup \limits_{L\leq 0,L\in 
\mathbb{Z}
}\sum \limits_{l=0}^{\infty }2^{\left( \lambda -\alpha _{\infty }\right)
lq_{1}}\sum \limits_{k=-\infty }^{L}2^{\left( \alpha \left( 0\right)
k-L\lambda \right) q_{1}},
\end{eqnarray*}%
which gives%
\begin{eqnarray*}
&\lesssim &\left \Vert b\right \Vert _{BMO}^{q_{1}}\left( \Psi +\sup
\limits_{L\leq 0,L\in 
\mathbb{Z}
}\sum \limits_{l=L}^{-1}2^{-l\lambda q_{1}}\left \vert \lambda _{l}\right
\vert ^{q_{1}}2^{\left( l-L\right) \lambda q_{1}}\sum \limits_{k=-\infty
}^{l}2^{\left( k-l\right) \alpha \left( 0\right) q_{1}}+\Psi \right) \\
&\lesssim &\left \Vert b\right \Vert _{BMO}^{q_{1}}\left( \Psi +\Psi \sup
\limits_{L\leq 0,L\in 
\mathbb{Z}
}\sum \limits_{l=L}^{-1}2^{\left( l-L\right) \lambda q_{1}}\sum
\limits_{k=-\infty }^{l}2^{\left( k-l\right) \alpha \left( 0\right)
q_{1}}\right) \\
&\lesssim &\left \Vert b\right \Vert _{BMO}^{q_{1}}\Psi \text{ }\left(
\alpha _{\infty }>\lambda \right) .
\end{eqnarray*}%
Secondly, by (\ref{12}), (3.6) in \cite{Gurbuz1} and the hypothesis $\beta
-n\delta _{2}<\alpha \left( 0\right) $, we get 
\begin{eqnarray*}
X_{2} &=&\sup \limits_{L\leq 0,L\in 
\mathbb{Z}
}2^{-L\lambda q_{1}}\sum \limits_{k=-\infty }^{L}2^{kq_{1}\alpha \left(
0\right) }\left( \sum \limits_{l=-\infty }^{k-1}\left \vert \lambda
_{l}\right \vert \left \Vert \left( \left[ b,I^{\beta }\right] \left(
a_{l}\right) \right) \chi _{k}\right \Vert _{L^{p_{2}(\cdot )}}\right)
^{q_{1}} \\
&\lesssim &\sup \limits_{L\leq 0,L\in 
\mathbb{Z}
}2^{-L\lambda q_{1}}\sum \limits_{k=-\infty }^{L}2^{kq_{1}\alpha \left(
0\right) }\left( \sum \limits_{l=-\infty }^{k-1}\left \Vert b\right \Vert
_{BMO}^{q_{1}}\left( k-l\right) ^{q_{1}}\left \vert \lambda _{l}\right \vert
^{q_{1}}2^{\left[ \left( \beta -n\delta _{2}\right) \left( k-l\right)
-\alpha _{l}l\right] q_{1}}\right) \\
&\lesssim &\left \Vert b\right \Vert _{BMO}^{q_{1}}\sup \limits_{L\leq
0,L\in 
\mathbb{Z}
}2^{-L\lambda q_{1}}\sum \limits_{k=-\infty }^{L}2^{kq_{1}\alpha \left(
0\right) }\left( \sum \limits_{l=-\infty }^{k-1}\left( k-l\right)
^{q_{1}}\left \vert \lambda _{l}\right \vert ^{q_{1}}2^{\left[ \left( \beta
-n\delta _{2}\right) \left( k-l\right) -\alpha \left( 0\right) l\right]
q_{1}}\right) \\
&\lesssim &\left \Vert b\right \Vert _{BMO}^{q_{1}}\sup \limits_{L\leq
0,L\in 
\mathbb{Z}
}2^{-L\lambda q_{1}}\sum \limits_{l=-\infty }^{L}\left \vert \lambda
_{j}\right \vert ^{q_{1}}\sum \limits_{k=l+1}^{-1}\left( k-l\right)
^{q_{1}}2^{\left( \beta -n\delta _{2}-\alpha \left( 0\right) \right) \left(
k-l\right) q_{1}} \\
&\lesssim &\left \Vert b\right \Vert _{BMO}^{q_{1}}\Psi .
\end{eqnarray*}%
Thirdly, by (\ref{10}) and (3.6) in \cite{Gurbuz1}, we obtain that 
\begin{eqnarray*}
Y_{1} &=&\sum \limits_{k=-\infty }^{-1}2^{kq_{1}\alpha \left( 0\right)
}\left( \sum \limits_{l=k}^{\infty }\left \vert \lambda _{l}\right \vert
\left \Vert \left( \left[ b,I^{\beta }\right] \left( a_{l}\right) \right)
\chi _{k}\right \Vert _{L^{p_{2}(\cdot )}}\right) ^{q_{1}} \\
&\lesssim &\sum \limits_{k=-\infty }^{-1}2^{kq_{1}\alpha \left( 0\right)
}\left( \sum \limits_{l=k}^{\infty }\left \vert \lambda _{l}\right \vert
\left \Vert b\right \Vert _{BMO}2^{-l\alpha _{l}}\right) ^{q_{1}} \\
&\lesssim &\left \Vert b\right \Vert _{BMO}^{q_{1}}\sum \limits_{k=-\infty
}^{-1}2^{kq_{1}\alpha \left( 0\right) }\left( \sum \limits_{l=k}^{-1}\left
\vert \lambda _{l}\right \vert ^{q_{1}}2^{-l\alpha \left( 0\right)
q_{1}}+\sum \limits_{l=0}^{\infty }\left \vert \lambda _{l}\right \vert
^{q_{1}}2^{-l\alpha _{\infty }q_{1}}\right) \\
&\lesssim &\left \Vert b\right \Vert _{BMO}^{q_{1}}\left( \sum
\limits_{k=-\infty }^{-1}\sum \limits_{l=k}^{-1}\left \vert \lambda
_{l}\right \vert ^{q_{1}}2^{\left( k-l\right) \alpha \left( 0\right)
q_{1}}+\sum \limits_{k=-\infty }^{-1}2^{k\alpha \left( 0\right) q_{1}}\sum
\limits_{l=0}^{\infty }\left \vert \lambda _{l}\right \vert
^{q_{1}}2^{-l\alpha _{\infty }q_{1}}\right) \\
&\lesssim &\left \Vert b\right \Vert _{BMO}^{q_{1}}\left( \sum
\limits_{l=-\infty }^{-1}\left \vert \lambda _{l}\right \vert ^{q_{1}}\sum
\limits_{k=\infty }^{l}2^{\left( k-l\right) \alpha \left( 0\right)
q_{1}}+\sum \limits_{l=0}^{\infty }\left \vert \lambda _{l}\right \vert
^{q_{1}}2^{-l\alpha _{\infty }q_{1}}\sum \limits_{k=-\infty }^{-1}2^{k\alpha
\left( 0\right) q_{1}}\right) \\
&\lesssim &\left \Vert b\right \Vert _{BMO}^{q_{1}}\left( \sum
\limits_{l=-\infty }^{-1}\left \vert \lambda _{l}\right \vert ^{q_{1}}+\sum
\limits_{l=0}^{\infty }\left \vert \lambda _{l}\right \vert
^{q_{1}}2^{-l\alpha _{\infty }q_{1}}2^{-l\lambda q_{1}}\sum
\limits_{k=-\infty }^{-1}2^{k\alpha \left( 0\right) q_{1}}\right) \\
&\lesssim &\left \Vert b\right \Vert _{BMO}^{q_{1}}\left( \Psi +\Psi \sum
\limits_{i=-\infty }^{l}\left \vert \lambda _{i}\right \vert ^{q_{1}}\sum
\limits_{l=0}^{\infty }2^{\left( \lambda -\alpha _{\infty }\right)
q_{1}l}\sum \limits_{k=-\infty }^{l}2^{k\alpha \left( 0\right) q_{1}}\right)
\\
&\lesssim &\left \Vert b\right \Vert _{BMO}^{q_{1}}\Psi \text{ }\left(
\alpha _{\infty }>\lambda \right) .
\end{eqnarray*}%
Fourthly, by (\ref{12}), (3.6) in \cite{Gurbuz1} and the assumption $\beta
-n\delta _{2}<\alpha \left( 0\right) $, we obtain that%
\begin{eqnarray*}
Y_{2} &=&\sum \limits_{k=-\infty }^{-1}2^{kq_{1}\alpha \left( 0\right)
}\left( \sum \limits_{l=-\infty }^{k-1}\left \vert \lambda _{l}\right \vert
\left \Vert \left( \left[ b,I^{\beta }\right] \left( a_{l}\right) \right)
\chi _{k}\right \Vert _{L^{p_{2}(\cdot )}}\right) ^{q_{1}} \\
&\lesssim &\left \Vert b\right \Vert _{BMO}^{q_{1}}\sum \limits_{k=-\infty
}^{-1}2^{kq_{1}\alpha \left( 0\right) }\left( \sum \limits_{l=-\infty
}^{k-1}\left \vert \lambda _{l}\right \vert ^{q_{1}}\left( k-l\right)
^{q_{1}}2^{\left[ \left( \beta -n\delta _{2}\right) \left( k-l\right)
-\alpha _{l}l\right] q_{1}}\right) \\
&\lesssim &\left \Vert b\right \Vert _{BMO}^{q_{1}}\sum \limits_{k=-\infty
}^{-1}\left \vert \lambda _{l}\right \vert ^{q_{1}}\sum
\limits_{k=l+1}^{-1}\left( k-l\right) ^{q_{1}}2^{\left( \beta -n\delta
_{2}-\alpha \left( 0\right) \right) \left( k-l\right) q_{1}} \\
&\lesssim &\left \Vert b\right \Vert _{BMO}^{q_{1}}\Psi .
\end{eqnarray*}%
Fifthly, similar to $X_{1}$, using (\ref{10}) and (3.6) in \cite{Gurbuz1},
we have%
\begin{eqnarray*}
Z_{1} &=&\sup \limits_{L>0,L\in 
\mathbb{Z}
}2^{-L\lambda q_{1}}\sum \limits_{k=0}^{L}2^{kq_{1}\alpha _{\infty }}\left(
\sum \limits_{l=k}^{\infty }\left \vert \lambda _{l}\right \vert \left \Vert
\left( \left[ b,I^{\beta }\right] \left( a_{l}\right) \right) \chi
_{k}\right \Vert _{L^{p_{2}(\cdot )}}\right) ^{q_{1}} \\
&\lesssim &\sup \limits_{L>0,L\in 
\mathbb{Z}
}2^{-L\lambda q_{1}}\sum \limits_{k=0}^{L}2^{kq_{1}\alpha _{\infty }}\sum
\limits_{l=k}^{\infty }\left \vert \lambda _{l}\right \vert ^{q_{1}}\left
\Vert \left( \left[ b,I^{\beta }\right] \left( a_{l}\right) \right) \chi
_{k}\right \Vert _{L^{p_{2}(\cdot )}}^{q_{1}} \\
&\lesssim &\left \Vert b\right \Vert _{BMO}^{q_{1}}\sup \limits_{L>0,L\in 
\mathbb{Z}
}2^{-L\lambda q_{1}}\sum \limits_{k=0}^{L}2^{kq_{1}\alpha _{\infty }}\sum
\limits_{l=k}^{\infty }\left \vert \lambda _{l}\right \vert
^{q_{1}}2^{-l\alpha _{l}q_{1}} \\
&\lesssim &\left \Vert b\right \Vert _{BMO}^{q_{1}}\sup \limits_{L>0,L\in 
\mathbb{Z}
}2^{-L\lambda q_{1}}\sum \limits_{k=0}^{L}2^{kq_{1}\alpha _{\infty }}\sum
\limits_{l=k}^{\infty }\left \vert \lambda _{l}\right \vert
^{q_{1}}2^{-l\alpha _{\infty }q_{1}} \\
&=&\left \Vert b\right \Vert _{BMO}^{q_{1}}\sup \limits_{L>0,L\in 
\mathbb{Z}
}2^{-L\lambda q_{1}}\sum \limits_{l=0}^{L}\left \vert \lambda _{l}\right
\vert ^{q_{1}}\sum \limits_{k=0}^{l}2^{\left( k-l\right) q_{1}\alpha
_{\infty }} \\
&&+\left \Vert b\right \Vert _{BMO}^{q_{1}}\sup \limits_{L>0,L\in 
\mathbb{Z}
}2^{-L\lambda q_{1}}\sum \limits_{l=L}^{\infty }\left \vert \lambda
_{l}\right \vert ^{q_{1}}\sum \limits_{k=0}^{L}2^{\left( k-l\right)
q_{1}\alpha _{\infty }} \\
&\lesssim &\left \Vert b\right \Vert _{BMO}^{q_{1}}\sup \limits_{L>0,L\in 
\mathbb{Z}
}2^{-L\lambda q_{1}}\sum \limits_{l=0}^{L}\left \vert \lambda _{l}\right
\vert ^{q_{1}} \\
&&+\left \Vert b\right \Vert _{BMO}^{q_{1}}\sup \limits_{L>0,L\in 
\mathbb{Z}
}\sum \limits_{l=L}^{\infty }2^{\left( l-L\right) q_{1}\lambda }2^{-l\lambda
q_{1}}\dsum \limits_{i=-\infty }^{l}\left \vert \lambda _{i}\right \vert
^{q_{1}}\sum \limits_{k=0}^{L}2^{\left( k-l\right) q_{1}\alpha _{\infty }},
\end{eqnarray*}%
which implies%
\begin{eqnarray*}
&\lesssim &\left \Vert b\right \Vert _{BMO}^{q_{1}}\Psi +\left \Vert b\right
\Vert _{BMO}^{q_{1}}\Psi \sup \limits_{L>0,L\in 
\mathbb{Z}
}\sum \limits_{l=L}^{\infty }2^{\left( l-L\right) \lambda q_{1}}2^{\left(
L-l\right) q_{1}\alpha _{\infty }} \\
&\lesssim &\left \Vert b\right \Vert _{BMO}^{q_{1}}\Psi +\left \Vert b\right
\Vert _{BMO}^{q_{1}}\Psi \sup \limits_{L>0,L\in 
\mathbb{Z}
}\sum \limits_{l=L}^{\infty }2^{\left( l-L\right) q_{1}\left( \lambda
-\alpha _{\infty }\right) } \\
&\lesssim &\left \Vert b\right \Vert _{BMO}^{q_{1}}\Psi \text{ }\left(
\alpha _{\infty }>\lambda \right) .
\end{eqnarray*}%
Finally, by (\ref{12}), (3.6) in \cite{Gurbuz1} and the assumptions that $%
\beta -n\delta _{2}<\alpha \left( 0\right) $ and $\beta -n\delta _{2}<\alpha
_{\infty }$, we get%
\begin{eqnarray*}
Z_{2} &=&\sup \limits_{L>0,L\in 
\mathbb{Z}
}2^{-L\lambda q_{1}}\sum \limits_{k=0}^{L}2^{kq_{1}\alpha _{\infty }}\left(
\sum \limits_{l=-\infty }^{k-1}\left \vert \lambda _{l}\right \vert \left
\Vert \left( \left[ b,I^{\beta }\right] \left( a_{l}\right) \right) \chi
_{k}\right \Vert _{L^{p_{2}(\cdot )}}\right) ^{q_{1}} \\
&\lesssim &\left \Vert b\right \Vert _{BMO}^{q_{1}}\sup \limits_{L>0,L\in 
\mathbb{Z}
}2^{-L\lambda q_{1}}\sum \limits_{k=0}^{L}2^{kq_{1}\alpha _{\infty }}\left(
\sum \limits_{l=-\infty }^{k-1}\left \vert \lambda _{l}\right \vert
^{q_{1}}\left( k-l\right) ^{q_{1}}2^{\left[ \left( \beta -n\delta
_{2}\right) \left( k-l\right) -\alpha _{l}l\right] q_{1}}\right) \\
&=&\left \Vert b\right \Vert _{BMO}^{q_{1}}\sup \limits_{L>0,L\in 
\mathbb{Z}
}2^{-L\lambda q_{1}}\sum \limits_{k=0}^{L}2^{kq_{1}\alpha _{\infty }}\left(
\sum \limits_{l=-\infty }^{-1}\left \vert \lambda _{l}\right \vert
^{q_{1}}\left( k-l\right) ^{q_{1}}2^{\left[ \left( \beta -n\delta
_{2}\right) \left( k-l\right) -\alpha \left( 0\right) l\right] q_{1}}\right)
\\
&&+\left \Vert b\right \Vert _{BMO}^{q_{1}}\sup \limits_{L>0,L\in 
\mathbb{Z}
}2^{-L\lambda q_{1}}\sum \limits_{k=0}^{L}2^{kq_{1}\alpha _{\infty }}\left(
\sum \limits_{l=0}^{k-1}\left \vert \lambda _{l}\right \vert ^{q_{1}}\left(
k-l\right) ^{q_{1}}2^{\left[ \left( \beta -n\delta _{2}\right) \left(
k-l\right) -\alpha _{\infty }l\right] q_{1}}\right) \\
&\lesssim &\left \Vert b\right \Vert _{BMO}^{q_{1}}\sup \limits_{L>0,L\in 
\mathbb{Z}
}2^{-L\lambda q_{1}}\sum \limits_{k=0}^{L}\left( k-l\right) ^{q_{1}}2^{kq_{1}%
\left[ \alpha _{\infty }-\left( \beta -n\delta _{2}\right) \right] }\sum
\limits_{l=-\infty }^{-1}\left \vert \lambda _{l}\right \vert ^{q_{1}}2^{%
\left[ \beta -n\delta _{2}-\alpha \left( 0\right) \right] lq_{1}} \\
&&+\left \Vert b\right \Vert _{BMO}^{q_{1}}\sup \limits_{L>0,L\in 
\mathbb{Z}
}2^{-L\lambda q_{1}}\sum \limits_{k=0}^{L}\left \vert \lambda _{l}\right
\vert ^{q_{1}}\sum \limits_{k=l+1}^{\infty }\left( k-l\right) ^{q_{1}}2^{%
\left[ \left( \beta -n\delta _{2}-\alpha _{\infty }\right) \left( l-k\right) %
\right] q_{1}} \\
&\lesssim &\left \Vert b\right \Vert _{BMO}^{q_{1}}\left( \sup
\limits_{L>0,L\in 
\mathbb{Z}
}2^{-L\lambda q_{1}}\sum \limits_{l=-\infty }^{-1}\left \vert \lambda
_{l}\right \vert ^{q_{1}}+\sup \limits_{L>0,L\in 
\mathbb{Z}
}2^{-L\lambda q_{1}}\sum \limits_{l=0}^{L-1}\left \vert \lambda _{l}\right
\vert ^{q_{1}}\right) \\
&\lesssim &\left \Vert b\right \Vert _{BMO}^{q_{1}}\Psi .
\end{eqnarray*}%
\textbf{Case 2: }$\left( 1<q_{1}<\infty \right) $. In this case, similar to
Case 1, we always exchange order of summation and use the convergence of a
geometric series, but use (\ref{3}) instead of (3.6) in \cite{Gurbuz1}.

Indeed, for $X_{1}$, let $\frac{1}{q_{1}}+\frac{1}{q_{1}^{\prime }}=1$. By (%
\ref{3}) and (\ref{10}), we have%
\begin{eqnarray*}
X_{1} &=&\sup \limits_{L\leq 0,L\in 
\mathbb{Z}
}2^{-L\lambda q_{1}}\sum \limits_{k=-\infty }^{L}2^{kq_{1}\alpha \left(
0\right) }\left( \sum \limits_{l=k}^{\infty }\left \vert \lambda _{l}\right
\vert \left \Vert \left( \left[ b,I^{\beta }\right] \left( a_{l}\right)
\right) \chi _{k}\right \Vert _{L^{p_{2}(\cdot )}}\right) ^{q_{1}} \\
&\lesssim &\left \Vert b\right \Vert _{BMO}^{q_{1}}\sup \limits_{L\leq
0,L\in 
\mathbb{Z}
}2^{-L\lambda q_{1}}\sum \limits_{k=-\infty }^{L}2^{kq_{1}\alpha \left(
0\right) }\left( \sum \limits_{l=k}^{\infty }\left \vert \lambda _{l}\right
\vert 2^{-l\alpha _{l}}\right) ^{q_{1}} \\
&\lesssim &\left \Vert b\right \Vert _{BMO}^{q_{1}}\sup \limits_{L\leq
0,L\in 
\mathbb{Z}
}2^{-L\lambda q_{1}}\sum \limits_{k=-\infty }^{L}\left( \sum
\limits_{l=k}^{-1}\left \vert \lambda _{l}\right \vert 2^{\left( k-l\right)
\alpha \left( 0\right) }\right) ^{q_{1}} \\
&&+\left \Vert b\right \Vert _{BMO}^{q_{1}}\sup \limits_{L\leq 0,L\in 
\mathbb{Z}
}2^{-L\lambda q_{1}}\sum \limits_{k=-\infty }^{L}2^{k\alpha \left( 0\right)
q_{1}}\left( \sum \limits_{l=0}^{\infty }\left \vert \lambda _{l}\right
\vert 2^{-l\alpha _{\infty }}\right) ^{q_{1}} \\
&\lesssim &\left \Vert b\right \Vert _{BMO}^{q_{1}}\sup \limits_{L\leq
0,L\in 
\mathbb{Z}
}2^{-L\lambda q_{1}}\sum \limits_{k=-\infty }^{L}\left( \sum
\limits_{l=k}^{-1}\left \vert \lambda _{l}\right \vert ^{q_{1}}2^{\frac{%
\left( k-l\right) \alpha \left( 0\right) q_{1}}{2}}\right) \left( \sum
\limits_{l=k}^{-1}2^{\frac{\left( k-l\right) \alpha \left( 0\right)
q_{1}^{\prime }}{2}}\right) ^{\frac{q_{1}}{q_{1}^{\prime }}} \\
&&+\left \Vert b\right \Vert _{BMO}^{q_{1}}\sup \limits_{L\leq 0,L\in 
\mathbb{Z}
}2^{-L\lambda q_{1}}\sum \limits_{k=-\infty }^{L}2^{kq_{1}\alpha \left(
0\right) }\left( \sum \limits_{l=0}^{\infty }\left \vert \lambda _{l}\right
\vert ^{q_{1}}2^{\frac{-l\alpha _{\infty }q_{1}}{2}}\right) \left( \sum
\limits_{l=0}^{\infty }2^{\frac{-l\alpha _{\infty }q_{1}^{\prime }}{2}%
}\right) ^{\frac{q_{1}}{q_{1}^{\prime }}},
\end{eqnarray*}%
which gives%
\begin{eqnarray*}
&\lesssim &\left \Vert b\right \Vert _{BMO}^{q_{1}}\sup \limits_{L\leq
0,L\in 
\mathbb{Z}
}2^{-L\lambda q_{1}}\sum \limits_{k=-\infty }^{L}\sum
\limits_{l=k}^{-1}\left \vert \lambda _{l}\right \vert ^{q_{1}}2^{\frac{%
\left( k-l\right) \alpha \left( 0\right) q_{1}}{2}} \\
&&+\left \Vert b\right \Vert _{BMO}^{q_{1}}\sup \limits_{L\leq 0,L\in 
\mathbb{Z}
}2^{-L\lambda q_{1}}\sum \limits_{k=-\infty }^{L}2^{kq_{1}\alpha \left(
0\right) }\sum \limits_{l=0}^{\infty }\left \vert \lambda _{l}\right \vert
^{q_{1}}2^{\frac{-l\alpha _{\infty }q_{1}}{2}} \\
&\lesssim &\left \Vert b\right \Vert _{BMO}^{q_{1}}\sup \limits_{L\leq
0,L\in 
\mathbb{Z}
}2^{-L\lambda q_{1}}\sum \limits_{l=-\infty }^{-1}\left \vert \lambda
_{l}\right \vert ^{q_{1}}\sum \limits_{k=-\infty }^{l}2^{\frac{\left(
k-l\right) \alpha \left( 0\right) q_{1}}{2}} \\
&&+\left \Vert b\right \Vert _{BMO}^{q_{1}}\sup \limits_{L\leq 0,L\in 
\mathbb{Z}
}\sum \limits_{l=0}^{\infty }2^{-l\lambda q_{1}}\left \vert \lambda
_{l}\right \vert ^{q_{1}}2^{\left( \lambda -\frac{\alpha _{\infty }}{2}%
\right) lq_{1}}2^{-L\lambda q_{1}}\sum \limits_{k=-\infty }^{L}2^{k\alpha
\left( 0\right) q_{1}} \\
&\leq &\left \Vert b\right \Vert _{BMO}^{q_{1}}\left( \sup \limits_{L\leq
0,L\in 
\mathbb{Z}
}2^{-L\lambda q_{1}}\sum \limits_{l=-\infty }^{L}\left \vert \lambda
_{l}\right \vert ^{q_{1}}+\sup \limits_{L\leq 0,L\in 
\mathbb{Z}
}2^{-L\lambda q_{1}}\sum \limits_{l=L}^{-1}\left \vert \lambda _{l}\right
\vert ^{q_{1}}\sum \limits_{k=-\infty }^{l}2^{\frac{\left( k-l\right) \alpha
\left( 0\right) q_{1}}{2}}\right) \\
&&+\left \Vert b\right \Vert _{BMO}^{q_{1}}\Psi \sup \limits_{L\leq 0,L\in 
\mathbb{Z}
}\sum \limits_{l=0}^{\infty }2^{\left( \lambda -\frac{\alpha _{\infty }}{2}%
\right) lq_{1}}\sum \limits_{k=-\infty }^{L}2^{\left( \alpha \left( 0\right)
k-L\lambda \right) q_{1}} \\
&\lesssim &\left \Vert b\right \Vert _{BMO}^{q_{1}}\left( \Psi +\sup
\limits_{L\leq 0,L\in 
\mathbb{Z}
}\sum \limits_{l=L}^{-1}2^{-l\lambda q_{1}}\left \vert \lambda _{l}\right
\vert ^{q_{1}}2^{\left( l-L\right) \lambda q_{1}}\sum \limits_{k=-\infty
}^{l}2^{\frac{\left( k-l\right) \alpha \left( 0\right) q_{1}}{2}}+\Psi
\right) \\
&\lesssim &\left \Vert b\right \Vert _{BMO}^{q_{1}}\left( \Psi +\Psi \sup
\limits_{L\leq 0,L\in 
\mathbb{Z}
}\sum \limits_{l=L}^{-1}2^{\left( l-L\right) \lambda q_{1}}\sum
\limits_{k=-\infty }^{l}2^{\frac{\left( k-l\right) \alpha \left( 0\right)
q_{1}}{2}}\right) \\
&\lesssim &\left \Vert b\right \Vert _{BMO}^{q_{1}}\Psi \text{ }\left(
\alpha _{\infty }>2\lambda \right) .
\end{eqnarray*}%
For $X_{2}$, let $\frac{1}{q_{1}}+\frac{1}{q_{1}^{\prime }}=1$. By (\ref{3}%
), (\ref{12}) and the assumption $\beta -n\delta _{2}<\alpha \left( 0\right) 
$, we get%
\begin{eqnarray*}
X_{2} &=&\sup \limits_{L\leq 0,L\in 
\mathbb{Z}
}2^{-L\lambda q_{1}}\sum \limits_{k=-\infty }^{L}2^{kq_{1}\alpha \left(
0\right) }\left( \sum \limits_{l=-\infty }^{k-1}\left \vert \lambda
_{l}\right \vert \left \Vert \left( \left[ b,I^{\beta }\right] \left(
a_{l}\right) \right) \chi _{k}\right \Vert _{L^{p_{2}(\cdot )}}\right)
^{q_{1}} \\
&\lesssim &\left \Vert b\right \Vert _{BMO}^{q_{1}}\sup \limits_{L\leq
0,L\in 
\mathbb{Z}
}2^{-L\lambda q_{1}}\sum \limits_{k=-\infty }^{L}2^{kq_{1}\alpha \left(
0\right) }\left( \sum \limits_{l=-\infty }^{k-1}\left \vert \lambda
_{l}\right \vert ^{q_{1}}\left( k-l\right) ^{q_{1}}2^{\left[ \left( \beta
-n\delta _{2}\right) \left( k-l\right) -\alpha _{l}l\right] q_{1}}\right) \\
&\lesssim &\left \Vert b\right \Vert _{BMO}^{q_{1}}\sup \limits_{L\leq
0,L\in 
\mathbb{Z}
}2^{-L\lambda q_{1}}\sum \limits_{k=-\infty }^{L}2^{kq_{1}\alpha \left(
0\right) }\left( \sum \limits_{l=-\infty }^{k-1}\left \vert \lambda
_{l}\right \vert ^{q_{1}}\left( k-l\right) ^{q_{1}}2^{\left[ \left( \beta
-n\delta _{2}\right) \left( k-l\right) -\alpha \left( 0\right) l\right] 
\frac{q_{1}}{2}}\right) \\
&&\times \left( \sum \limits_{l=-\infty }^{k-1}2^{\left[ \left( \beta
-n\delta _{2}\right) \left( k-l\right) -\alpha \left( 0\right) l\right] 
\frac{q_{1}^{\prime }}{2}}\right) ^{\frac{q_{1}}{q_{1}^{\prime }}} \\
&\lesssim &\left \Vert b\right \Vert _{BMO}^{q_{1}}\sup \limits_{L\leq
0,L\in 
\mathbb{Z}
}2^{-L\lambda q_{1}}\sum \limits_{k=-\infty }^{L}2^{kq_{1}\alpha \left(
0\right) }\left( \sum \limits_{l=-\infty }^{k-1}\left \vert \lambda
_{l}\right \vert ^{q_{1}}\left( k-l\right) ^{q_{1}}2^{\left[ \left( \beta
-n\delta _{2}\right) \left( k-l\right) -\alpha \left( 0\right) l\right] 
\frac{q_{1}}{2}}\right) \\
&\lesssim &\left \Vert b\right \Vert _{BMO}^{q_{1}}\sup \limits_{L\leq
0,L\in 
\mathbb{Z}
}2^{-L\lambda q_{1}}\sum \limits_{l=-\infty }^{L}\left \vert \lambda
_{l}\right \vert ^{q_{1}}\sum \limits_{k=l+1}^{-1}2^{\left[ \left( \beta
-n\delta _{2}-\alpha \left( 0\right) \right) \left( k-l\right) \right] \frac{%
q_{1}}{2}} \\
&\lesssim &\left \Vert b\right \Vert _{BMO}^{q_{1}}\Psi .
\end{eqnarray*}%
For $Y_{1}$, let $\frac{1}{q_{1}}+\frac{1}{q_{1}^{\prime }}=1$. Then,
applying (\ref{3}) and (\ref{10}), we know that%
\begin{eqnarray*}
Y_{1} &=&\sum \limits_{k=-\infty }^{-1}2^{kq_{1}\alpha \left( 0\right)
}\left( \sum \limits_{l=k}^{\infty }\left \vert \lambda _{l}\right \vert
\left \Vert \left( \left[ b,I^{\beta }\right] \left( a_{l}\right) \right)
\chi _{k}\right \Vert _{L^{p_{2}(\cdot )}}\right) ^{q_{1}} \\
&\lesssim &\left \Vert b\right \Vert _{BMO}^{q_{1}}\sum \limits_{k=-\infty
}^{-1}2^{kq_{1}\alpha \left( 0\right) }\left( \sum \limits_{l=k}^{\infty
}\left \vert \lambda _{l}\right \vert 2^{-l\alpha _{l}}\right) ^{q_{1}} \\
&\lesssim &\left \Vert b\right \Vert _{BMO}^{q_{1}}\left[ \sum
\limits_{k=-\infty }^{-1}\left( \sum \limits_{l=k}^{-1}\left \vert \lambda
_{l}\right \vert 2^{\left( k-l\right) \alpha \left( 0\right) }\right)
^{q_{1}}+\sum \limits_{k=-\infty }^{-1}2^{kq_{1}\alpha \left( 0\right)
}\left( \sum \limits_{l=0}^{\infty }\left \vert \lambda _{l}\right \vert
2^{-l\alpha _{\infty }}\right) ^{q_{1}}\right] \\
&\lesssim &\left \Vert b\right \Vert _{BMO}^{q_{1}}\sum \limits_{k=-\infty
}^{-1}\left( \sum \limits_{l=k}^{-1}\left \vert \lambda _{l}\right \vert
^{q_{1}}2^{\left( k-l\right) \alpha \left( 0\right) \frac{q_{1}}{2}}\right)
\left( \sum \limits_{l=k}^{-1}2^{\left( k-l\right) \alpha \left( 0\right) 
\frac{q_{1}^{\prime }}{2}}\right) ^{\frac{q_{1}}{q_{1}^{\prime }}} \\
&&+\left \Vert b\right \Vert _{BMO}^{q_{1}}\sum \limits_{k=-\infty
}^{-1}2^{kq_{1}\alpha \left( 0\right) }\left( \sum \limits_{l=0}^{\infty
}\left \vert \lambda _{l}\right \vert ^{q_{1}}2^{\left( -l\alpha _{\infty
}\right) \frac{q_{1}}{2}}\right) \left( \sum \limits_{l=0}^{\infty
}2^{\left( -l\alpha _{\infty }\right) \frac{q_{1}^{\prime }}{2}}\right) ^{%
\frac{q_{1}}{q_{1}^{\prime }}} \\
&\lesssim &\left \Vert b\right \Vert _{BMO}^{q_{1}}\left( \sum
\limits_{l=-\infty }^{-1}\left \vert \lambda _{l}\right \vert ^{q_{1}}\sum
\limits_{k=-\infty }^{l}2^{\left( k-l\right) \alpha \left( 0\right) \frac{%
q_{1}}{2}}+\sum \limits_{k=-\infty }^{-1}2^{kq_{1}\alpha \left( 0\right)
}\sum \limits_{l=0}^{\infty }\left \vert \lambda _{l}\right \vert
^{q_{1}}2^{-l\alpha _{\infty }\frac{q_{1}}{2}}\right) \\
&\lesssim &\left \Vert b\right \Vert _{BMO}^{q_{1}}\left( \sum
\limits_{l=-\infty }^{-1}\left \vert \lambda _{l}\right \vert ^{q_{1}}+\sum
\limits_{l=0}^{\infty }2^{\left( \lambda -\frac{\alpha _{\infty }}{2}\right)
lq_{1}}2^{-l\lambda q_{1}}\sum \limits_{i=-\infty }^{j}\left \vert \lambda
_{i}\right \vert ^{q_{1}}\sum \limits_{k=-\infty }^{-1}2^{kq_{1}\alpha
\left( 0\right) }\right) \\
&\lesssim &\left \Vert b\right \Vert _{BMO}^{q_{1}}\left( \Psi +\Psi \sum
\limits_{l=0}^{\infty }2^{\left( \lambda -\frac{\alpha _{\infty }}{2}\right)
lq_{1}}\sum \limits_{k=-\infty }^{-1}2^{k\alpha \left( 0\right) q_{1}}\right)
\\
&\lesssim &\left \Vert b\right \Vert _{BMO}^{q_{1}}\Psi \text{ }\left(
\alpha _{\infty }>\lambda \right) .
\end{eqnarray*}%
For $Y_{2}$, let $\frac{1}{q_{1}}+\frac{1}{q_{1}^{\prime }}=1$. Then, (\ref%
{3}), (\ref{12}) and the assumption $\beta -n\delta _{2}<\alpha \left(
0\right) $, we have%
\begin{eqnarray*}
Y_{2} &=&\sum \limits_{k=-\infty }^{-1}2^{kq_{1}\alpha \left( 0\right)
}\left( \sum \limits_{l=-\infty }^{k-1}\left \vert \lambda _{l}\right \vert
\left \Vert \left( \left[ b,I^{\beta }\right] \left( a_{l}\right) \right)
\chi _{k}\right \Vert _{L^{p_{2}(\cdot )}}\right) ^{q_{1}} \\
&\lesssim &\left \Vert b\right \Vert _{BMO}^{q_{1}}\sum \limits_{k=-\infty
}^{-1}2^{kq_{1}\alpha \left( 0\right) }\left( \sum \limits_{l=-\infty
}^{k-1}\left \vert \lambda _{l}\right \vert \left( k-l\right) 2^{\left(
\beta -n\delta _{2}\right) \left( k-l\right) -\alpha \left( 0\right)
l}\right) ^{q_{1}} \\
&\lesssim &\left \Vert b\right \Vert _{BMO}^{q_{1}}\sum \limits_{k=-\infty
}^{-1}2^{kq_{1}\alpha \left( 0\right) }\left( \sum \limits_{l=-\infty
}^{k-1}\left \vert \lambda _{l}\right \vert ^{q_{1}}2^{\left[ \left( \beta
-n\delta _{2}\right) \left( k-l\right) -\alpha \left( 0\right) l\right] 
\frac{q_{1}}{2}}\right) \\
&&\times \left( \sum \limits_{l=-\infty }^{k-1}\left( k-l\right) 2^{\left[
\left( \beta -n\delta _{2}\right) \left( k-l\right) -\alpha \left( 0\right) l%
\right] \frac{q_{1}^{\prime }}{2}}\right) ^{\frac{q_{1}}{q_{1}^{\prime }}} \\
&\lesssim &\left \Vert b\right \Vert _{BMO}^{q_{1}}\sum \limits_{k=-\infty
}^{-1}2^{kq_{1}\alpha \left( 0\right) }\left( \sum \limits_{l=-\infty
}^{k-1}\left \vert \lambda _{l}\right \vert ^{q_{1}}2^{\left[ \left( \beta
-n\delta _{2}\right) \left( k-l\right) -\alpha \left( 0\right) l\right] 
\frac{q_{1}}{2}}\right) ,
\end{eqnarray*}%
which implies%
\begin{eqnarray*}
&\lesssim &\left \Vert b\right \Vert _{BMO}^{q_{1}}\sum \limits_{l=-\infty
}^{-1}\left \vert \lambda _{l}\right \vert ^{q_{1}}\sum
\limits_{k=l+1}^{-1}2^{\left( \beta -n\delta _{2}-\alpha \left( 0\right)
\right) \left( k-l\right) \frac{q_{1}}{2}} \\
&\lesssim &\left \Vert b\right \Vert _{BMO}^{q_{1}}\Psi .
\end{eqnarray*}%
For $Z_{1}$, let $\frac{1}{q_{1}}+\frac{1}{q_{1}^{\prime }}=1$. Then, using (%
\ref{3}) and (\ref{10}), we have%
\begin{eqnarray*}
Z_{1} &=&\sup \limits_{L>0,L\in 
\mathbb{Z}
}2^{-L\lambda q_{1}}\sum \limits_{k=0}^{L}2^{kq_{1}\alpha _{\infty }}\left(
\sum \limits_{l=k}^{\infty }\left \vert \lambda _{l}\right \vert \left \Vert
\left( \left[ b,I^{\beta }\right] \left( a_{l}\right) \right) \chi
_{k}\right \Vert _{L^{p_{2}(\cdot )}}\right) ^{q_{1}} \\
&\lesssim &\sup \limits_{L>0,L\in 
\mathbb{Z}
}2^{-L\lambda q_{1}}\sum \limits_{k=0}^{L}2^{kq_{1}\alpha _{\infty }}\left(
\sum \limits_{l=k}^{\infty }\left \vert \lambda _{l}\right \vert
^{q_{1}}\left \Vert \left( \left[ b,I^{\beta }\right] \left( a_{l}\right)
\right) \chi _{k}\right \Vert _{L^{p_{2}(\cdot )}}^{\frac{q_{1}}{2}}\right)
\left( \sum \limits_{l=k}^{\infty }\left \Vert \left( \left[ b,I^{\beta }%
\right] \left( a_{l}\right) \right) \chi _{k}\right \Vert _{L^{p_{2}(\cdot
)}}^{\frac{q_{1}^{\prime }}{2}}\right) ^{\frac{q_{1}}{q_{1}^{\prime }}} \\
&\lesssim &\left \Vert b\right \Vert _{BMO}^{q_{1}}\sup \limits_{L>0,L\in 
\mathbb{Z}
}2^{-L\lambda q_{1}}\sum \limits_{k=0}^{L}2^{kq_{1}\alpha _{\infty }}\left(
\sum \limits_{l=k}^{\infty }\left \vert \lambda _{l}\right \vert
^{q_{1}}\left \Vert a_{l}\right \Vert _{L^{p_{1}(\cdot )}}^{\frac{q_{1}}{2}%
}\right) \left( \sum \limits_{l=k}^{\infty }\left \Vert a_{l}\right \Vert
_{L^{p_{1}(\cdot )}}^{\frac{q_{1}^{\prime }}{2}}\right) ^{\frac{q_{1}}{%
q_{1}^{\prime }}} \\
&\lesssim &\left \Vert b\right \Vert _{BMO}^{q_{1}}\sup \limits_{L>0,L\in 
\mathbb{Z}
}2^{-L\lambda q_{1}}\sum \limits_{k=0}^{L}2^{kq_{1}\alpha _{\infty }}\left(
\sum \limits_{l=k}^{\infty }\left \vert \lambda _{l}\right \vert
^{q_{1}}\left \vert B_{l}\right \vert ^{-\frac{\alpha _{l}q_{1}}{\left(
2n\right) }}\right) \left( \sum \limits_{l=k}^{\infty }\left \vert
B_{l}\right \vert ^{-\frac{\alpha _{l}q_{1}^{\prime }}{\left( 2n\right) }%
}\right) ^{\frac{q_{1}}{q_{1}^{\prime }}} \\
&\lesssim &\left \Vert b\right \Vert _{BMO}^{q_{1}}\sup \limits_{L>0,L\in 
\mathbb{Z}
}2^{-L\lambda q_{1}}\sum \limits_{k=0}^{L}2^{k\frac{q_{1}}{2}\alpha _{\infty
}}\left( \sum \limits_{l=k}^{\infty }\left \vert \lambda _{l}\right \vert
^{q_{1}}\left \vert B_{l}\right \vert ^{-\frac{\alpha _{l}q_{1}}{\left(
2n\right) }}\right) \\
&=&\left \Vert b\right \Vert _{BMO}^{q_{1}}\left( \sup \limits_{L>0,L\in 
\mathbb{Z}
}2^{-L\lambda q_{1}}\sum \limits_{l=0}^{L}\left \vert \lambda _{l}\right
\vert ^{q_{1}}\sum \limits_{k=0}^{l}2^{\left( k-l\right) \frac{q_{1}}{2}%
\alpha _{\infty }}+\sup \limits_{L>0,L\in 
\mathbb{Z}
}2^{-L\lambda q_{1}}\sum \limits_{l=L}^{\infty }\left \vert \lambda
_{l}\right \vert ^{q_{1}}\sum \limits_{k=0}^{L}2^{\left( k-l\right) \frac{%
q_{1}}{2}\alpha _{\infty }}\right) \\
&\lesssim &\left \Vert b\right \Vert _{BMO}^{q_{1}}\left( \sup
\limits_{L>0,L\in 
\mathbb{Z}
}2^{-L\lambda q_{1}}\sum \limits_{l=0}^{L}\left \vert \lambda _{l}\right
\vert ^{q_{1}}+\sup \limits_{L>0,L\in 
\mathbb{Z}
}\sum \limits_{l=L}^{\infty }2^{\left( l-L\right) q_{1}\lambda }2^{-l\lambda
q_{1}}\dsum \limits_{i=-\infty }^{l}\left \vert \lambda _{i}\right \vert
^{q_{1}}\sum \limits_{k=0}^{L}2^{\left( k-l\right) \frac{q_{1}}{2}\alpha
_{\infty }}\right) \\
&\lesssim &\left \Vert b\right \Vert _{BMO}^{q_{1}}\left( \Psi +\Psi \sup
\limits_{L>0,L\in 
\mathbb{Z}
}\sum \limits_{l=L}^{\infty }2^{\left( l-L\right) \lambda q_{1}}2^{\left(
L-l\right) \frac{q_{1}}{2}\alpha _{\infty }}\right) \\
&\lesssim &\left \Vert b\right \Vert _{BMO}^{q_{1}}\left( \Psi +\Psi \sup
\limits_{L>0,L\in 
\mathbb{Z}
}\sum \limits_{l=L}^{\infty }2^{\left( l-L\right) \lambda q_{1}\left(
\lambda -\frac{\alpha _{\infty }}{2}\right) }\right) \\
&\lesssim &\left \Vert b\right \Vert _{BMO}^{q_{1}}\Psi \text{ }\left(
\alpha _{\infty }>2\lambda \right) .
\end{eqnarray*}%
Finally, we consider the term $Z_{2}$. Now, let $\frac{1}{q_{1}}+\frac{1}{%
q_{1}^{\prime }}=1$. By (\ref{3}), (\ref{12}) and the conditions $\beta
-n\delta _{2}<\alpha \left( 0\right) $ and $\beta -n\delta _{2}<\alpha
_{\infty }$, we obtain that%
\begin{eqnarray*}
Z_{2} &=&\sup \limits_{L>0,L\in 
\mathbb{Z}
}2^{-L\lambda q_{1}}\sum \limits_{k=0}^{L}2^{kq_{1}\alpha _{\infty }}\left(
\sum \limits_{l=-\infty }^{k-1}\left \vert \lambda _{l}\right \vert \left
\Vert \left( \left[ b,I^{\beta }\right] \left( a_{l}\right) \right) \chi
_{k}\right \Vert _{L^{p_{2}(\cdot )}}\right) ^{q_{1}} \\
&\lesssim &\left \Vert b\right \Vert _{BMO}^{q_{1}}\sup \limits_{L>0,L\in 
\mathbb{Z}
}2^{-L\lambda q_{1}}\sum \limits_{k=0}^{L}2^{kq_{1}\alpha _{\infty }}\left(
\sum \limits_{l=-\infty }^{k-1}\left \vert \lambda _{l}\right \vert \left(
k-l\right) 2^{\left( \beta -n\delta _{2}\right) \left( k-l\right) -\alpha
_{l}l}\right) ^{q_{1}} \\
&\lesssim &\left \Vert b\right \Vert _{BMO}^{q_{1}}\sup \limits_{L>0,L\in 
\mathbb{Z}
}2^{-L\lambda q_{1}}\sum \limits_{k=0}^{L}2^{kq_{1}\alpha _{\infty }}\left(
\sum \limits_{l=-\infty }^{-1}\left \vert \lambda _{l}\right \vert \left(
k-l\right) 2^{\left( \beta -n\delta _{2}\right) \left( k-l\right) -\alpha
\left( 0\right) l}\right) ^{q_{1}}
\end{eqnarray*}%
\begin{eqnarray*}
&&+\left \Vert b\right \Vert _{BMO}^{q_{1}}\sup \limits_{L>0,L\in 
\mathbb{Z}
}2^{-L\lambda q_{1}}\sum \limits_{k=0}^{L}2^{kq_{1}\alpha _{\infty }}\left(
\sum \limits_{l=0}^{k-1}\left \vert \lambda _{l}\right \vert \left(
k-l\right) 2^{\left( \beta -n\delta _{2}\right) \left( k-l\right) -\alpha
_{\infty }l}\right) ^{q_{1}} \\
&\lesssim &\left \Vert b\right \Vert _{BMO}^{q_{1}}\sup \limits_{L>0,L\in 
\mathbb{Z}
}2^{-L\lambda q_{1}}\sum \limits_{k=0}^{L}2^{kq_{1}\left[ \alpha _{\infty
}-\left( \beta -n\delta _{2}\right) \right] }\left( \sum \limits_{l=-\infty
}^{-1}\left \vert \lambda _{l}\right \vert \left( k-l\right) 2^{l\left[
\left( \beta -n\delta _{2}\right) -\alpha \left( 0\right) \right] }\right)
^{q_{1}} \\
&&+\left \Vert b\right \Vert _{BMO}^{q_{1}}\sup \limits_{L>0,L\in 
\mathbb{Z}
}2^{-L\lambda q_{1}}\sum \limits_{k=0}^{L}\left( \sum
\limits_{l=0}^{k-1}\left \vert \lambda _{l}\right \vert \left( k-l\right)
2^{\left( \beta -n\delta _{2}-\alpha _{\infty }\right) \left( k-l\right)
}\right) ^{q_{1}} \\
&\lesssim &\left \Vert b\right \Vert _{BMO}^{q_{1}}\left( \sup
\limits_{L>0,L\in 
\mathbb{Z}
}2^{-L\lambda q_{1}}\sum \limits_{l=-\infty }^{-1}\left \vert \lambda
_{l}\right \vert ^{q_{1}}2^{\left[ \left( \beta -n\delta _{2}\right) -\alpha
\left( 0\right) \right] l\frac{q_{1}}{2}}\right) \left( \sum
\limits_{l=-\infty }^{-1}\left( k-l\right) ^{q_{1}}2^{\left[ \left( \beta
-n\delta _{2}\right) -\alpha \left( 0\right) \right] l\frac{q_{1}^{\prime }}{%
2}}\right) ^{\frac{q_{1}}{q_{1}^{\prime }}} \\
&&+\left \Vert b\right \Vert _{BMO}^{q_{1}}\sup \limits_{L>0,L\in 
\mathbb{Z}
}2^{-L\lambda q_{1}}\sum \limits_{k=0}^{L}\left( \sum
\limits_{l=0}^{k-1}\left \vert \lambda _{l}\right \vert ^{q_{1}}\left(
k-l\right) ^{q_{1}}2^{\left[ \left( \beta -n\delta _{2}-\alpha _{\infty
}\right) \left( k-l\right) \right] \frac{q_{1}}{2}}\right) \\
&&\times \left( \sum \limits_{l=0}^{k-1}2^{\left[ \left( \beta -n\delta
_{2}-\alpha _{\infty }\right) \left( k-l\right) \right] \frac{q_{1}^{\prime }%
}{2}}\right) ^{\frac{q_{1}}{q_{1}^{\prime }}} \\
&\lesssim &\left \Vert b\right \Vert _{BMO}^{q_{1}}\sup \limits_{L>0,L\in 
\mathbb{Z}
}2^{-L\lambda q_{1}}\sum \limits_{l=-\infty }^{-1}\left \vert \lambda
_{l}\right \vert ^{q_{1}}\left( k-l\right) ^{q_{1}}2^{\left[ \left( \beta
-n\delta _{2}\right) -\alpha \left( 0\right) \right] l\frac{q_{1}}{2}} \\
&&+\left \Vert b\right \Vert _{BMO}^{q_{1}}\sup \limits_{L>0,L\in 
\mathbb{Z}
}2^{-L\lambda q_{1}}\sum \limits_{k=0}^{L}\sum \limits_{l=0}^{k-1}\left
\vert \lambda _{l}\right \vert ^{q_{1}}\left( k-l\right) ^{q_{1}}2^{\left[
\left( \beta -n\delta _{2}-\alpha _{\infty }\right) \left( k-l\right) \right]
\frac{q_{1}}{2}} \\
&\lesssim &\left \Vert b\right \Vert _{BMO}^{q_{1}}\sup \limits_{L>0,L\in 
\mathbb{Z}
}2^{-L\lambda q_{1}}\sum \limits_{l=-\infty }^{-1}\left \vert \lambda
_{l}\right \vert ^{q_{1}} \\
&&+\left \Vert b\right \Vert _{BMO}^{q_{1}}\sup \limits_{L>0,L\in 
\mathbb{Z}
}2^{-L\lambda q_{1}}\sum \limits_{l=0}^{L-1}\left \vert \lambda _{l}\right
\vert ^{q_{1}}\sum \limits_{k=l+1}^{L}\left( k-l\right) ^{q_{1}}2^{\left[
\left( \beta -n\delta _{2}-\alpha _{\infty }\right) \left( k-l\right) \right]
\frac{q_{1}}{2}} \\
&\lesssim &\left \Vert b\right \Vert _{BMO}^{q_{1}}\left( \sup
\limits_{L>0,L\in 
\mathbb{Z}
}2^{-L\lambda q_{1}}\sum \limits_{l=-\infty }^{-1}\left \vert \lambda
_{l}\right \vert ^{q_{1}}+\sup \limits_{L>0,L\in 
\mathbb{Z}
}2^{-L\lambda q_{1}}\sum \limits_{l=0}^{L-1}\left \vert \lambda _{l}\right
\vert ^{q_{1}}\right) \\
&\lesssim &\left \Vert b\right \Vert _{BMO}^{q_{1}}\Psi .
\end{eqnarray*}%
Then, by joining the above inequalities for $X,Y$ and $Z$, we obtain (\ref%
{100}). Thus, the proof is completed.
\end{proof}


\begin{thebibliography}{99}
\bibitem{Capone} C. Capone, D. Cruz-Uribe, A. Fiorenza. The fractional
maximal operator and fractional integrals on variable $L^{p}$\ spaces. Rev.
Mat. Iberoam. 23(3), 743--770 (2007).

\bibitem{Fefferman and Stein} C. Fefferman, E.M. Stein. $H^{p}$ spaces of
several variables. Acta Math. 129(3-4), 137--193 (1972).
doi:10.1007/BF02392215

\bibitem{Gurbuz1} F. G\"{u}rb\"{u}z. Some inequalities for Riesz Potential
on homogeneous variable exponent Herz-Morrey-Hardy spaces. (submitted.)
arXiv:2411.13880v1 [math.FA] 21 Nov 2024 DOI: 10.48550/arXiv.2411.13880
(accepted for publication in the Springer Volume)

\bibitem{Herz} C.S. Herz. Lipschitz spaces and Bernstein's theorem on
absolutely convergent Fourier transforms. J. Math. Mech. 18(4), 283--323
(1968/69).

\bibitem{Izuki} M. Izuki. Boundedness of sublinear operators on Herz spaces
with variable exponent and application to wavelet characterization. Anal.
Math. 36(1), 33--50 (2010). doi: 10.1007/s10476-010-0102-8

\bibitem{Izuki1} M. Izuki. Boundedness of commutators on Herz spaces with
variable exponent. Rend. Circ. Mat. Palermo (2) 59(2) (2010), 199-213.

\bibitem{Kovacik} O. Kov\'{a}\v{c}ik, J. R\'{a}kosn\'{\i}k. On spaces $%
L^{p\left( x\right) }$\ and $W^{k,p\left( x\right) }$. Czechoslovak Math. J.
41, 592-618 (1991). doi:10.21136/CMJ.1991.102493

\bibitem{Stein} E.M. Stein. Harmonic Analysis. Princeton University Press,
Princeton, NJ, USA (1993).

\bibitem{Xu} J. Xu, X.Yang. Herz-Morrey-Hardy spaces with variable exponents
and their applications. J. Funct. Spaces. 2015, Art. ID 160635, 1-19 (2015).
doi: 10.1155/2015/160635

\bibitem{XU} J. Xu, X. Yang. The molecular decomposition of
Herz-Morrey-Hardy spaces with variable exponents and its application. J.
Math. Inequal. 10, 977--1008 (2016). doi:10.7153/jmi-10-79
\end{thebibliography}
\end{document}